\documentclass[10pt,reqno]{amsart}

\usepackage[T1]{fontenc}
\usepackage{lmodern}
\usepackage{microtype}
\usepackage{mathtools}
\usepackage{amssymb}
\usepackage{mathrsfs}
\usepackage{esint}
\usepackage[
  letterpaper,
  left=0.90in,
  right=0.90in,
  top=0.85in,
  bottom=0.90in
]{geometry}
\usepackage{xcolor}
\usepackage[
  colorlinks=true,
  linkcolor=blue,
  citecolor=blue,
  urlcolor=blue,
  hypertexnames=false
]{hyperref}
\usepackage[nameinlink,capitalize,noabbrev]{cleveref}

\allowdisplaybreaks
\numberwithin{equation}{section}

\newtheorem{theorem}{Theorem}[section]
\newtheorem{proposition}[theorem]{Proposition}
\newtheorem{lemma}[theorem]{Lemma}
\newtheorem{corollary}[theorem]{Corollary}

\theoremstyle{definition}

\newtheorem{example}[theorem]{Example}
\theoremstyle{remark}
\newtheorem{remark}[theorem]{Remark}

\newcommand{\C}{\mathbb C}
\newcommand{\N}{\mathbb N}
\newcommand{\dd}{\,d}
\newcommand{\Hol}{\mathcal H}
\newcommand{\loc}{\mathrm{loc}}
\newcommand{\dist}{\operatorname{dist}}
\newcommand{\diag}{\operatorname{diag}}
\newcommand{\supp}{\operatorname{supp}}
\newcommand{\IDA}{\operatorname{IDA}}
\newcommand{\Piabs}{\Pi}
\newcommand{\Xk}{\mathfrak D_{p,q}^{\,r}}
\newcommand{\xk}{\mathfrak d_{p,q}^{\,r}}
\newcommand{\avg}{\mathop{\fint}}

\newcommand{\papertitle}{Absolutely summing Toeplitz operators on weighted Fock spaces}
\title[Absolutely summing Toeplitz operators]{\papertitle}
\author{Chunxu Xu and Jianxiang Dong}
\date{}

\hypersetup{
  pdftitle={Absolutely summing Toeplitz operators on weighted Fock spaces},
  pdfauthor={Chunxu Xu and Jianxiang Dong},
  pdfsubject={Absolutely summing operators on weighted Fock spaces},
  pdfkeywords={absolutely summing operator, Toeplitz operator, weighted Fock space, IDA symbol, diagonal ideal}
}

\begin{document}

\thispagestyle{plain}
\begin{center}
  {\Large\bfseries\MakeUppercase{\papertitle}\par}
  \vspace{1.15em}
  {\normalsize Chunxu Xu${}^{a}$, Jianxiang Dong${}^{b,*}$\par}
  \vspace{0.75em}
  {\small\itshape
  ${}^{a}$School of Science, Nanjing Forestry University,
  Nanjing 210037, P.R. China\par
  ${}^{b}$School of Mathematics and Statistics, Tianshui Normal University,
  Tianshui 741000, P.R. China\par}
\end{center}

\vspace{0.7em}
\begin{quotation}
\small
\noindent\textsc{Abstract.}
We study absolutely $r$-summing Toeplitz operators
$T_f:F_\varphi^p\to F_\varphi^q$ with complex symbols.
The weight has a uniformly positive and bounded real Hessian.
The main difficulty is that local distance to holomorphic functions does
not detect an entire part of the symbol, while the Toeplitz operator does.
Using the complete classification of absolutely summing diagonal maps, we
characterize absolutely summing Fock--Carleson embeddings.  An IDA
decomposition, a Rademacher argument, and local scalar recovery then detect
the missing analytic part of a complex symbol.  We obtain equivalent
criteria in terms of the local $L^q$ mean, the complex ball average, and the
Berezin transform for every $1\le p,q,r<\infty$.  At $q=1$ the criterion is
$L^{s_p}$, where $s_p=2p/(3p-2)$ for $p\le2$ and $s_p=1$ for $p>2$.
We also obtain a sectorial-symbol criterion without an IDA assumption and
prove sharpness by separated models.
\end{quotation}



\medskip

\begingroup
\renewcommand{\thefootnote}{*}
\footnotetext[1]{Corresponding author.\newline
Email addresses:
\href{mailto:1968385450@qq.com}{1968385450@qq.com} (Chunxu Xu),
\href{mailto:jianxd@tsnu.edu.cn}{jianxd@tsnu.edu.cn} (Jianxiang Dong).}
\endgroup

\markboth{CHUNXU XU AND JIANXIANG DONG}
{ABSOLUTELY SUMMING TOEPLITZ OPERATORS}

\section{Introduction}

\subsection{Motivation and the problem}

Let $v$ be Lebesgue measure on $\C^n$.  We identify $\C^n$ with
$\mathbb R^{2n}$.  Let $\varphi\in C^2(\C^n)$ have a uniformly positive
and bounded real Hessian.  This is the usual Hessian matrix on
$\mathbb R^{2n}$.  For an open set $U$, write
$\Hol(U)=\{h:h\text{ is holomorphic on }U\}$.  The space
$L^s_{\loc}(U)$ consists of the functions that belong to $L^s$ on every
compact subset of $U$.  We use the weighted spaces
\[
L_\varphi^s=L^s(\C^n,e^{-s\varphi}\dd v),
\qquad
F_\varphi^s=L_\varphi^s\cap\Hol(\C^n).
\]
Let $P:L_\varphi^2\to F_\varphi^2$ be the orthogonal Fock projection.
Let $K$ be its reproducing kernel.  Put
\[
k_z(w)=\frac{K(w,z)}{K(z,z)^{1/2}},
\qquad
\Gamma=\operatorname{span}\{k_z:z\in\C^n\}.
\]
Thus $k_z$ is the normalized kernel.  Define
\begin{equation}\label{eq:symbol-domain-intro}
\mathcal D_\varphi
=\{f\text{ measurable}:fk_z\in L_\varphi^1
\text{ for every }z\in\C^n\}.
\end{equation}
For $f\in\mathcal D_\varphi$, set on $\Gamma$
\[
M_fg=fg,
\qquad
T_fg=P(fg),
\qquad
H_fg=(I-P)(fg).
\]
These are the standard Segal--Bargmann and Fock-space conventions
\cite{Folland1989,Zhu2012}.

For positive symbols, local masses and Fock--Carleson measures govern
boundedness and compactness
\cite{HuLv2011,HuLv2014,IsralowitzZhu2010,
IsralowitzVirtanenWolf2015,SchusterVarolin2012}.
For complex symbols, cancellation is an extra difficulty.  BMO and mean
oscillation give useful criteria
\cite{Bauer2005,BauerCoburnIsralowitz2010,CoburnIsralowitzLi2011}.
Localization gives a complementary operator-theoretic approach
\cite{HuLvWick2018}.  However, the operator $T_f$ does not directly recover
the positive measure $|f|^q\dd v$.  The complex cancellations remain.

Local distance to holomorphic functions measures the nonanalytic part of a
symbol.  This method starts with Luecking's work \cite{Luecking1992}.  Hu
and Virtanen developed it on Fock spaces \cite{HuVirtanen2023}.  Let
$B(z,\delta)$ be the Euclidean ball with center $z$ and radius $\delta$.
With
$\avg_B u\dd v=|B|^{-1}\int_Bu\dd v$, put
\[
G_{q,\delta}f(z)
=
\inf_{h\in\Hol(B(z,\delta))}
\left(
\avg_{B(z,\delta)}|f(w)-h(w)|^q\dd v(w)
\right)^{1/q}.
\]
The quantity $G_{q,\delta}f(z)$ is the local $L^q$-distance from $f$ to
holomorphic functions on $B(z,\delta)$.  It is called the integral distance
to analytic functions, or IDA.  We also use
\[
M_{q,\delta}f(z)
=
\left(
\avg_{B(z,\delta)}|f(w)|^q\dd v(w)
\right)^{1/q},
\qquad
a_\delta f(z)=\avg_{B(z,\delta)}f(w)\dd v(w).
\]
We call $M_{q,\delta}f$ the normalized local $q$-mean of $f$.
It measures the local size of $f$.  We call $a_\delta f$ the complex ball
average of $f$.  It keeps the phase and cancellation of a complex symbol.
Thus $G_{q,\delta}f$ measures the nonanalytic part, while
$M_{q,\delta}f$ and $a_\delta f$ provide two forms of local scalar data.
Fock-space Hankel ideals have also been
studied through Hankel forms, symmetric ideals, mean oscillation, and the
Berger--Coburn phenomenon
\cite{JansonPeetreRochberg1987,Farnsworth2011,HuVirtanen2022,
HuVirtanenCorr2023}.  Mixed-exponent and compactness results appear in
\cite{HuWang2018,HaggerVirtanen2021}.  IDA gives one framework for the
nonanalytic part of these results.

There is a basic obstruction.  If $h$ is entire, then
\[
G_{q,\delta}(f+h)=G_{q,\delta}f,
\qquad
T_{f+h}=T_f+M_h.
\]
Thus IDA does not see the entire part of a symbol, but the Toeplitz operator
does.  For example, $G_{q,\delta}z_1=0$, while multiplication by $z_1$ is
unbounded on the classical Fock space.  The purpose of this paper is to
recover this missing analytic part by scalar data.  We use a complex ball
average or the Berezin transform.

Schatten ideals are described by singular values \cite{Simon2005}.
Absolute summability is different.  It is defined through weakly summable
finite sequences.  Its foundations are given in
\cite{Pietsch1967,LindenstraussPelczynski1968,Pietsch1980,
DiestelJarchowTonge1995}.  For a Banach space $X$, write
$B_X=\{x\in X:\|x\|_X\le1\}$.  A bounded map $A:X\to Y$ is $r$-summing if
\[
\left(\sum_{j=1}^m\|Ax_j\|_Y^r\right)^{1/r}
\le C
\sup_{x^*\in B_{X^*}}
\left(\sum_{j=1}^m|x^*(x_j)|^r\right)^{1/r}
\]
for every finite family $(x_j)$.  The least possible constant is denoted by
$\pi_r(A)$.  The class of these maps is the operator ideal
$\Piabs_r(X,Y)$ \cite{DiestelJarchowTonge1995}.

Absolutely summing Carleson embeddings have been studied on Hardy,
Bergman, and Fock spaces
\cite{ChenHeWang2025,HeJreisLefevreLou2024,LefevreRodriguez2018}.
Chen, He, and Wang treated the same-exponent embedding on weighted Fock
spaces over $\C$ \cite{ChenHeWang2025}.  Wang and Shi obtained the
same-exponent theorem on generalized Fock spaces over $\C^n$
\cite{WangShi2026}.  Both results assume $1<p<\infty$.  The positive
Toeplitz theorem in \cite{HuWang2025} is also same-exponent.  The Hankel
theorem in \cite{HuLv2026} concerns the classical Fock space and maps
$F^p$ into $L^p$; its endpoint $q=1$ therefore has $p=1$.  None of these
results gives the full off-diagonal range $p\ne q$.  In particular, they
do not treat arbitrary $F_\varphi^p\to L_\varphi^1$, nor do they give a
complex-symbol Toeplitz criterion in the full range considered here.

Fan, He, Wang, and Zeng recently completed the classification of absolutely
summing diagonal maps, including the range left open by Garling, and applied
it to weighted Bergman Carleson embeddings and Hankel operators
\cite{FanHeWangZeng2026}.  We use their completed diagonal index as an
input.  Our main work begins after this sequence-space step.  We establish a
Fock--Carleson theorem for general weights with controlled real Hessian.  We
then treat complex Toeplitz symbols by combining IDA decomposition with an
operator-diagonal argument and local scalar recovery.  This produces the
Toeplitz--Hankel graph equivalence, the direct $q=1$ endpoint, and criteria in
terms of the complex ball average and the Berezin transform.  We also obtain
a criterion for sectorial symbols without a matching IDA assumption, as well
as compactness, analytic rigidity, and sharp separated models.

We now define the local space used in our main problem.  Fix
$1\le p,q,r<\infty$.  For a scalar sequence $b=(b_j)$, let
\[
D_b:\ell^p\longrightarrow\ell^q,
\qquad D_b(c_j)=(b_jc_j).
\]
Here $\ell^p$ is the usual scalar sequence space.  Define
\[
\mathfrak d_{p,q}^{\,r}
=\{b:D_b\in\Piabs_r(\ell^p,\ell^q)\},
\qquad
\|b\|_{\mathfrak d_{p,q}^{\,r}}=\pi_r(D_b).
\]
A $\delta$-lattice is a sequence $\{a_j\}$ whose $\delta$-balls cover
$\C^n$ and whose smaller $c\delta$-balls are pairwise disjoint for one
fixed $c\in(0,1)$.  Fix one such lattice and an associated measurable
partition $\C^n=\bigsqcup_jE_j$, with each $E_j$ contained in a fixed
multiple of $B(a_j,\delta)$ and containing a fixed smaller ball about
$a_j$.  For a measurable function $u$, set
\[
u_j^\#=\operatorname*{ess\,sup}_{z\in E_j}|u(z)|,
\qquad
\|u\|_{\mathfrak D_{p,q}^{\,r}}
=\|\{u_j^\#\}\|_{\mathfrak d_{p,q}^{\,r}}.
\]
The essential supremum is the smallest number that bounds $|u|$ almost
everywhere on the cell.  We call $\mathfrak D_{p,q}^{\,r}$ a
lattice-amalgam space because it combines local cell suprema by one global
sequence norm.  Different fixed admissible lattices give equivalent norms;
this is proved in \cref{lem:discretization}.

Our problem is to decide when the following equivalence holds:
\begin{equation*}
\boxed{\begin{gathered}
1\le p,q,r<\infty,\quad
f\in\mathcal D_\varphi\cap L^q_{\loc},\quad
G_{q,\delta}f\in\mathfrak D_{p,q}^{\,r},\\
T_f:F_\varphi^p\to F_\varphi^q\text{ is in }\Piabs_r
\quad\stackrel{?}{\Longleftrightarrow}\quad
M_{q,\delta}f\in\mathfrak D_{p,q}^{\,r}.
\end{gathered}}
\end{equation*}
The question includes three independent indices.  Fock-space inclusions do
not reduce it to $p=q$.  The same three-index structure already occurs for
the diagonal map $D_b$.  The ideal $\mathfrak d_{p,q}^{\,r}$ changes with
$(p,q,r)$.  One boundary case is a logarithmic Orlicz space.  We define it in
\cref{sec:diagonal-index}.  The endpoint $q=1$ needs a separate argument
because the usual reflexive duality is not available.  We instead prove a
direct $L^1$ estimate for a $\bar\partial$ solution.

\subsection{Main results}

Our starting point is the diagonal ideal and its lattice-amalgam realization
$\Xk$ defined above.  A Borel measure $\mu$ is locally finite when
$\mu(K)<\infty$ for every compact set $K$.  For a positive locally finite
Borel measure $\mu$, let
$J_\mu:F_\varphi^p\to L_\varphi^q(\mu)$ denote the embedding
$J_\mu g=g$.  We first prove
\begin{equation*}
J_\mu\in\Piabs_r(F_\varphi^p,L_\varphi^q(\mu))
\iff
\{\mu(B(a_j,\delta))^{1/q}\}_j\in\mathfrak d_{p,q}^{\,r}.
\end{equation*}
The proof is quantitative and uniform over the lattice.  It does not use an
explicit formula for the diagonal ideal.  Next, an IDA decomposition yields
\begin{equation*}
H_f\in\Piabs_r(F_\varphi^p,L_\varphi^q)
\iff G_{q,\delta}f\in\Xk,
\qquad 1\le q<\infty.
\end{equation*}
Finally, Rademacher extraction recovers the analytic component from the
operator diagonal.  Let $\iota:F_\varphi^q\hookrightarrow L_\varphi^q$
be the natural inclusion.  The identities
\begin{equation}\label{eq:basic-identities-intro}
M_f=\iota T_f+H_f,
\qquad
T_f=PM_f,
\qquad
H_f=(I-P)M_f,
\end{equation}
then complete the transference from multiplication to Toeplitz operators.

For $f\in\mathcal D_\varphi$, its Berezin transform is
\begin{equation*}
\widetilde f(z)
=\int_{\C^n}f(w)|k_z(w)|^2e^{-2\varphi(w)}\dd v(w).
\end{equation*}
For Banach spaces $X$ and $Y$, $X\oplus_qY$ denotes their product with
norm $(\|x\|_X^q+\|y\|_Y^q)^{1/q}$.

Define the matching IDA class by
\begin{equation}\label{eq:IDA-class-intro}
\IDA_{p,q}^{r}
=\{f\in\mathcal D_\varphi\cap L^q_{\loc}:G_{q,\delta}f\in\Xk\}.
\end{equation}
An \emph{IDA symbol of type $(p,q,r)$} means an element of this class.
The term \emph{matching} means that the IDA quantity is measured in the
same space $\mathfrak D_{p,q}^{\,r}$ as the local size of the symbol.
By \cref{lem:discretization}, this class is independent of the fixed
radius $\delta>0$.

\begin{theorem}\label{thm:intro-main}
Assume that $\varphi$ satisfies \eqref{eq:hessian}.  Let $1\le p,q,r<\infty$ and fix $\delta>0$.  Suppose that
\[
f\in\mathcal D_\varphi\cap L^q_{\loc}(\C^n).
\]
Define
\[
\mathscr T_fg=(T_fg,H_fg),
\qquad
\mathscr T_f:F_\varphi^p\longrightarrow
F_\varphi^q\oplus_qL_\varphi^q.
\]
The operators $M_f,T_f,H_f$, and $\mathscr T_f$ are initially defined on
$\Gamma$ as in \eqref{eq:symbol-domain-intro}.
Then the following conditions are equivalent:
\begin{enumerate}
\item $M_f\in\Piabs_r(F_\varphi^p,L_\varphi^q)$;
\item $\mathscr T_f\in\Piabs_r(F_\varphi^p,F_\varphi^q\oplus_qL_\varphi^q)$;
\item $T_f\in\Piabs_r(F_\varphi^p,F_\varphi^q)$ and
$H_f\in\Piabs_r(F_\varphi^p,L_\varphi^q)$;
\item $M_{q,\delta}f\in\Xk$;
\item $a_\delta f\in\Xk$ and $G_{q,\delta}f\in\Xk$.
\end{enumerate}
Whenever these conditions hold, the four core operators have unique bounded
extensions, and the identities in \eqref{eq:basic-identities-intro} hold on
all of $F_\varphi^p$.
Moreover,
\begin{equation}\label{eq:intro-main-norm}
\begin{aligned}
\pi_r(M_f)
&\asymp \pi_r(\mathscr T_f)
\asymp \pi_r(T_f)+\pi_r(H_f)\\
&\asymp \|M_{q,\delta}f\|_{\Xk}
\asymp
\|a_\delta f\|_{\Xk}+\|G_{q,\delta}f\|_{\Xk}.
\end{aligned}
\end{equation}
They are also equivalent to
\begin{enumerate}\setcounter{enumi}{5}
\item $\widetilde f\in\Xk$ and $G_{q,\delta}f\in\Xk$,
\end{enumerate}
and
\begin{equation}\label{eq:intro-main-scalar-norm}
\|M_{q,\delta}f\|_{\Xk}
\asymp \|\widetilde f\|_{\Xk}+\|G_{q,\delta}f\|_{\Xk}.
\end{equation}
The constants may depend on $p,q,r,n$, the Hessian bounds, the fixed radius
$\delta$, and fixed lattice separation and overlap parameters.  They are
independent of $f$, of ball centers, and of lattice translates.
\end{theorem}

\begin{corollary}\label{cor:intro-ida}
Under the assumptions of \cref{thm:intro-main}, assume that
$f\in\IDA_{p,q}^{r}$.  Then the
following conditions are equivalent:
\[
T_f\in\Piabs_r(F_\varphi^p,F_\varphi^q),\qquad
M_f\in\Piabs_r(F_\varphi^p,L_\varphi^q),\qquad
M_{q,\delta}f\in\Xk,
\]
and each is equivalent to $a_\delta f\in\Xk$ and to
$\widetilde f\in\Xk$.  If $D_f=\|G_{q,\delta}f\|_{\Xk}$, then
\begin{equation}\label{eq:intro-ida-norm}
\begin{split}
\pi_r(T_f)+D_f
&\asymp \pi_r(M_f)+D_f
\asymp \|M_{q,\delta}f\|_{\Xk}\\
&\asymp \|a_\delta f\|_{\Xk}+D_f
\asymp \|\widetilde f\|_{\Xk}+D_f.
\end{split}
\end{equation}
\end{corollary}

At $q=1$, the exceptional diagonal range is absent.  Put
\begin{equation}\label{eq:q-one-exponent-intro}
s_p=
\begin{cases}
\dfrac{2p}{3p-2},&1\le p\le2,\\[4pt]
1,&2<p<\infty.
\end{cases}
\end{equation}
Under $G_{1,\delta}f\in L^{s_p}$, the endpoint takes the explicit form
\begin{equation*}
T_f\in\Piabs_r(F_\varphi^p,F_\varphi^1)
\iff M_{1,\delta}f\in L^{s_p}
\iff a_\delta f\in L^{s_p}
\iff\widetilde f\in L^{s_p}.
\end{equation*}
The exponent $s_p$ is independent of $r$.

The graph theorem is unconditional.  The matching IDA hypothesis is needed
only for $T_f$ alone.  For every $1\le q<\infty$,
\[
G_{q,\delta}f\in\Xk
\quad\Longleftrightarrow\quad
H_f\in\Piabs_r.
\]
Together with \eqref{eq:basic-identities-intro}, this recovers $M_f$ from
$T_f$.  Without IDA, one still has
\begin{equation*}
T_f\in\Piabs_r
\quad\Longrightarrow\quad
\widetilde f\in\Xk,
\qquad
\|\widetilde f\|_{\Xk}\lesssim\pi_r(T_f).
\end{equation*}

There is one substantial class for which this necessary condition is also
sufficient.  Suppose that, for some $|\eta|=1$, $0\le\theta<\pi/2$, and
$C_f>0$,
\begin{equation}\label{eq:sector-preview}
\operatorname{Re}(\eta f)\ge (\cos\theta)|f|
\quad\text{a.e.},
\qquad
M_{q,\delta}f\le C_f\,a_\delta(|f|).
\end{equation}
Then
\begin{equation}\label{eq:sector-equivalence-preview}
T_f\in\Piabs_r
\iff M_{q,\delta}f\in\Xk
\iff a_\delta f\in\Xk
\iff \widetilde f\in\Xk.
\end{equation}
No IDA hypothesis occurs in \eqref{eq:sector-equivalence-preview}.  If
$q=1$, the second condition in \eqref{eq:sector-preview} is automatic,
and every occurrence of $\Xk$ in
\eqref{eq:sector-equivalence-preview} may be replaced by $L^{s_p}$.

Let $p'$ be the conjugate exponent, $1/p+1/p'=1$, with
$1'=\infty$.  Let
\begin{equation*}
\mathcal E
=\{(p,q,r):1<p<2<q<\infty,\ r>\max\{p',q\}\}.
\end{equation*}
Garling identified the diagonal index outside $\mathcal E$.  Fan, He, Wang,
and Zeng identified the remaining index in $\mathcal E$
\cite{FanHeWangZeng2026}.  Consequently, our Toeplitz criteria have an
explicit $L^s$, $L^\infty$, or logarithmic Orlicz form for every parameter
triple.  In the former exceptional range, the exponent is
\[
\kappa=\frac{p'q(r-2)}{(p'-2)(q-2)+2(r-2)}.
\]
The positive separated models in \cref{sec:sharpness} show that these
sequence exponents are sharp for the Fock Toeplitz problem.
Concrete index transitions appear in \cref{ex:concrete-indices}.

\subsection{Main ideas of the proof}

Let $R_j$ denote restriction to the $j$th lattice ball.  Let
$d=(d_j)$ be the sequence of local block norms.  The proof has four steps:
\begin{equation*}
\begin{gathered}
D_d\in\Piabs_r(\ell^p,\ell^q)
\Longrightarrow
\bigoplus_jR_j\in\Piabs_r
\Longrightarrow
J_\mu\in\Piabs_r,\\
G_{q,\delta}f
\Longrightarrow
f=f_1+f_2
\Longrightarrow
H_f=H_{f_1}+H_{f_2},\\
A\in\Piabs_r(F_\varphi^p,F_\varphi^q)
\Longrightarrow
\left\{\frac{(Ak_{a_j})(a_j)}{K(a_j,a_j)^{1/2}}\right\}_j
\in\mathfrak d_{p,q}^{\,r},\\
M_f=\iota T_f+H_f,
\qquad T_f=PM_f,
\qquad H_f=(I-P)M_f.
\end{gathered}
\end{equation*}
The first line uses Pietsch domination on local blocks.  The second line
splits the symbol into a smooth part and a locally small remainder.  A
kernel estimate for a $\bar\partial$ solution controls the smooth part.
This kernel is integrable, which also gives the case $q=1$.  The third line
uses Rademacher signs to remove the off-diagonal entries.  The last line
joins the multiplication, Toeplitz, and Hankel operators.  This order
separates the Fock-space estimates from the diagonal classification.

A Banach sequence lattice has an order-continuous norm if
$0\le b^{(m)}\downarrow0$ coordinatewise implies
$\|b^{(m)}\|\to0$.  In the concrete sequence spaces used here, this
property is equivalent to density of $c_{00}$, the finitely supported
scalar sequences.  In these regimes,
\begin{equation*}
c_{00}\text{ dense in }\mathfrak d_{p,q}^{\,r},\quad
M_f\in\Piabs_r
\quad\Longrightarrow\quad
M_f,\ T_f,\ H_f\text{ are compact}.
\end{equation*}
Under matching IDA, the same conclusion follows from $T_f\in\Piabs_r$.
Let $c_0$ denote the scalar sequences converging to zero.
If $f$ is entire, then
\begin{equation*}
T_f\in\Piabs_r
\quad\Longleftrightarrow\quad
\begin{cases}
f=0,&\mathfrak d_{p,q}^{\,r}\subset c_0,\\
f\text{ is constant},&\mathfrak d_{p,q}^{\,r}=\ell^\infty.
\end{cases}
\end{equation*}
We also treat compactly supported and sectorial symbols.  Separated bump
symbols show that the diagonal ideal is sharp.

Section~\ref{sec:preliminaries} fixes the Fock and diagonal-ideal notation.
Section~\ref{sec:carleson} proves the Fock--Carleson theorem.
Section~\ref{sec:hankel} develops IDA decomposition, Hankel estimates, and
scalar recovery.  The main theorem is proved in \cref{sec:toeplitz}.
Section~\ref{sec:consequences} contains explicit parameter forms,
compactness, rigidity, and sharpness.

\section{Preliminaries}\label{sec:preliminaries}

\subsection{Weighted Fock spaces}

Throughout the paper, $v$ denotes Lebesgue measure on $\C^n$.  If $X$ is a
Banach space, $B_X$ is its closed unit ball.  For a set $E$, $\mathbf1_E$
is its indicator and $\#E$ is its cardinality.  Set
$\N=\{1,2,\ldots\}$.  An entire function is a function holomorphic on all
of $\C^n$.  The notation $L^s_{\loc}(U)$ means local $L^s$ integrability:
the function belongs to $L^s(K)$ for every compact $K\subset U$.  For
$u\in C^1(\C^n)$, $|\bar\partial u|$ is the Euclidean norm of its
$(0,1)$-gradient.  We assume that
\begin{equation}\label{eq:hessian}
\varphi\in C^2(\C^n;\mathbb R),
\qquad
m_\varphi I_{2n}\le \operatorname{Hess}_{\mathbb R}\varphi(z)
\le M_\varphi I_{2n},
\end{equation}
where $0<m_\varphi\le M_\varphi<\infty$ and $I_{2n}$ is the identity
matrix on $\mathbb R^{2n}$.  The inequalities are inequalities of real
symmetric quadratic forms.  For $1\le s<\infty$, write
\[
L_\varphi^s=L^s(\C^n,e^{-s\varphi}\dd v),
\qquad
F_\varphi^s=L_\varphi^s\cap\Hol(\C^n).
\]
The norm is denoted by $\|g\|_{s,\varphi}=\|ge^{-\varphi}\|_{L^s(v)}$.
The notation $A\lesssim B$ means $A\le CB$.  The constant may depend on
the dimension, the Hessian bounds, and fixed parameters.  It is independent
of the symbol, the measure, and lattice translations.  We write
$A\asymp B$ when both estimates hold.
When a radius or a lattice occurs, the constant may depend on the fixed
radius and on the fixed separation and overlap parameters.  It is uniform
in the center of every ball and over translates of that lattice.  No
constant is asserted to be uniform as a radius tends to zero or infinity.

Let $K(z,w)$ be the reproducing kernel of $F_\varphi^2$, let
\[
k_z(w)=\frac{K(w,z)}{K(z,z)^{1/2}},
\qquad
\Gamma=\operatorname{span}\{k_z:z\in\C^n\},
\]
and let $P$ be the Fock projection.  On $L_\varphi^2$, it is the
orthogonal projection onto $F_\varphi^2$ and
\[
Pg(z)=\int_{\C^n}g(w)K(z,w)e^{-2\varphi(w)}\dd v(w).
\]
Under \eqref{eq:hessian}, $P$ extends boundedly from $L_\varphi^s$ onto
$F_\varphi^s$ for $1\le s<\infty$.  Moreover, there are constants
$c,C,\varepsilon>0$ such that
\begin{align}
K(z,z)e^{-2\varphi(z)}&\asymp1,\notag\\
|k_z(w)|e^{-\varphi(w)}&\le Ce^{-c|z-w|},\label{eq:kernel-upper}\\
|k_z(w)|e^{-\varphi(w)}&\ge C^{-1},\qquad |z-w|<\varepsilon.\label{eq:kernel-lower}
\end{align}
Also, for each fixed $s$,
\begin{equation}\label{eq:kernel-norm}
\|k_z\|_{s,\varphi}\asymp1.
\end{equation}
Moreover,
\begin{equation}\label{eq:kernel-span-dense}
\overline{\Gamma}^{\,F_\varphi^s}=F_\varphi^s,
\qquad 1\le s<\infty.
\end{equation}
These facts follow from the weighted kernel estimates in \cite{Delin1998,SchusterVarolin2012}; see also \cite{HuVirtanen2023}.
We write
\[
\iota:F_\varphi^s\hookrightarrow L_\varphi^s
\]
for the canonical isometric inclusion.

We use the initial symbol domain
\begin{equation*}
\mathcal D_\varphi
=
\{f\text{ measurable}:fk_z\in L_\varphi^1\text{ for every }z\in\C^n\}.
\end{equation*}
We call the dense space $\Gamma$ the common operator core.
For $f\in\mathcal D_\varphi$, the operators
\[
T_fg=P(fg),
\qquad
H_fg=(I-P)(fg),
\qquad g\in\Gamma,
\]
are well defined.  By \eqref{eq:kernel-span-dense}, every bounded extension
of a core operator on $\Gamma$ to all of $F_\varphi^p$ is unique.  Such extensions, when they
exist, are denoted by the same symbols.  The Berezin transform is
defined first by the absolutely convergent integral
\begin{equation*}
\widetilde f(z)
=\int_{\C^n}f(w)|k_z(w)|^2e^{-2\varphi(w)}\dd v(w).
\end{equation*}
Indeed, \eqref{eq:kernel-upper} gives
\[
\int_{\C^n}|f(w)||k_z(w)|^2e^{-2\varphi(w)}\dd v(w)
\lesssim \|fk_z\|_{1,\varphi}<\infty.
\]
The reproducing identity then yields
\begin{equation}\label{eq:berezin-evaluation}
\widetilde f(z)=\frac{(T_fk_z)(z)}{K(z,z)^{1/2}}.
\end{equation}
Throughout, a statement such as $T_f\in\Piabs_r$ means that the operator on
$\Gamma$ admits an $r$-summing extension to the displayed Fock space.  The
same convention is used for $H_f$, $M_f$, and $\mathscr T_f$.
For $M_f:F_\varphi^p\to L_\varphi^q$, this means explicitly that
$fg\in L_\varphi^q$ for every $g\in\Gamma$ and that
$g\mapsto fg$ admits the stated extension.  Thus a global
$L_\varphi^q$ product condition is not assumed before the criterion is
verified.
For $T_f$ and $H_f$, membership likewise includes that their core values
belong to the displayed target spaces.

The following local normalization is useful.  Its proof is included to make all constants uniform in the center of a ball.

\begin{lemma}\label{lem:gauge}
For each $R>0$ and $a\in\C^n$, there is a complex affine holomorphic function $L_a$ satisfying $L_a(a)=0$ and
\begin{equation*}
|\varphi(w)-\varphi(a)-\operatorname{Re}L_a(w)|\le C_R,
\qquad w\in B(a,R),
\end{equation*}
where $C_R$ is independent of $a$.
\end{lemma}

\begin{proof}
Let
\[
L_a(w)=2\sum_{j=1}^n\frac{\partial\varphi}{\partial z_j}(a)(w_j-a_j).
\]
Then $\operatorname{Re}L_a(w)=\nabla\varphi(a)\cdot(w-a)$.  Taylor's formula and the upper bound in \eqref{eq:hessian} give
\[
|\varphi(w)-\varphi(a)-\operatorname{Re}L_a(w)|
\le \frac{M_\varphi}{2}|w-a|^2
\le \frac{M_\varphi R^2}{2}.
\]
\end{proof}

It follows that, on each fixed ball $B(a,R)$,
\begin{equation}\label{eq:gauge-norm}
|g(w)|e^{-\varphi(w)}
\asymp
e^{-\varphi(a)}|g(w)e^{-L_a(w)}|.
\end{equation}
This reduces local weighted estimates to estimates on a fixed Euclidean ball.

\subsection{Absolutely summing ideals}

For a finite family $\boldsymbol x=(x_1,\ldots,x_m)$ in a Banach space $X$, put
\begin{equation*}
w_r(\boldsymbol x;X)
=
\sup_{x^*\in B_{X^*}}
\left(\sum_{j=1}^m|x^*(x_j)|^r\right)^{1/r}.
\end{equation*}
An operator $A:X\to Y$ belongs to $\Piabs_r(X,Y)$ precisely when
\begin{equation*}
\left(\sum_{j=1}^m\|Ax_j\|_Y^r\right)^{1/r}
\le \pi_r(A)w_r(\boldsymbol x;X)
\end{equation*}
for all $m$ and all $\boldsymbol x$.  We repeatedly use
\begin{align}
\|A\|&\le\pi_r(A),\notag\\
\pi_r(BAC)&\le\|B\|\pi_r(A)\|C\|,\notag\\
\pi_r(A_1+A_2)&\le\pi_r(A_1)+\pi_r(A_2).
\label{eq:summing-triangle}
\end{align}
If $A_j:X\to Y_j$, $j=1,2$, and
\[
A=(A_1,A_2):X\longrightarrow Y_1\oplus_qY_2,
\]
then the coordinate projections and the triangle inequality give
\begin{equation}\label{eq:finite-direct-sum}
\max\{\pi_r(A_1),\pi_r(A_2)\}
\le\pi_r(A)
\le\pi_r(A_1)+\pi_r(A_2).
\end{equation}
All integrals of $\Piabs_r$-valued maps are first taken on finite
dimensional truncations.  Equation~\eqref{eq:summing-triangle} and
monotone convergence then give the infinite dimensional operator.

\subsection{The intrinsic diagonal ideal and its explicit regimes}\label{sec:diagonal-index}

Let $1\le p,q\le\infty$ and $1\le r<\infty$.  For a scalar sequence $b=(b_j)$, let
\[
D_b:\ell^p\longrightarrow\ell^q,
\qquad
D_b(x_j)=(b_jx_j).
\]
For $1\le s<\infty$, the space $\ell^s$ consists of sequences $x$ with
$\|x\|_{\ell^s}=(\sum_j|x_j|^s)^{1/s}<\infty$.  The space $\ell^\infty$
consists of bounded sequences and has the supremum norm.
Define the diagonal ideal
\begin{equation}\label{eq:intrinsic-diagonal-ideal}
\mathfrak d_{p,q}^{\,r}
=\{b:D_b\in\Piabs_r(\ell^p,\ell^q)\},
\qquad
\|b\|_{\mathfrak d_{p,q}^{\,r}}=\pi_r(D_b).
\end{equation}
All Fock-space proofs below use this intrinsic ideal.  Garling
\cite{Garling1974} identified it except in one parameter range.  Fan, He,
Wang, and Zeng \cite{FanHeWangZeng2026} recently identified the remaining
range.  We use this complete classification only when we state explicit
function-space corollaries.  As usual, $p'$ is the
conjugate index, with $1'=\infty$ and $\infty'=1$.
To avoid ambiguity in Garling's notation, write $P$ for the parameter in
\cite[Theorem~9]{Garling1974}.  His diagonal map has the form
\[
D_b:\ell^{P'}\longrightarrow\ell^q.
\]
Our source space is $\ell^p$.  Hence
\begin{equation*}
P'=p,
\qquad
P=p'.
\end{equation*}
The formulas below follow from this substitution.  The two ranges left open
in Garling's theorem combine, in our notation, as
\[
\begin{aligned}
&2<q\le P<\infty,\quad r>P,\\
&2<P<q<\infty,\quad r>q,
\end{aligned}
\qquad\Longleftrightarrow\qquad
1<p<2<q<\infty,\quad r>\max\{p',q\}.
\]

For $1<q<2$, let $q^-$ denote the logarithmic Orlicz index defined as
follows.  Write $\log^+t=\max\{\log t,0\}$ and, for sufficiently small
$\varepsilon>0$, put
\[
M_q(t)=t^q\bigl(1+\log^+(t^{-1})\bigr),
\qquad 0\le t\le\varepsilon.
\]
Choose $\varepsilon$ so that $M_q$ is increasing and convex.  A Young
function is an increasing convex map $\Psi:[0,\infty)\to[0,\infty)$ with
$\Psi(0)=0$ and $\Psi(t)\to\infty$ as $t\to\infty$.  Extend $M_q$ to the
Young function
\[
\Psi_q(t)=
\begin{cases}
M_q(t),&0\le t\le\varepsilon,\\
M_q(\varepsilon)+M_q'(\varepsilon)(t-\varepsilon)+(t-\varepsilon)^2,
&t>\varepsilon.
\end{cases}
\]
We write $\ell^{q^-}=\ell^{\Psi_q}$ and $L^{q^-}=L^{\Psi_q}$, equipped
with their Luxemburg norms.
Thus
\begin{align*}
\|b\|_{\ell^{q^-}}
&=\inf\left\{\lambda>0:
\sum_j\Psi_q\!\left(\frac{|b_j|}{\lambda}\right)\le1\right\},\\
\|u\|_{L^{q^-}}
&=\inf\left\{\lambda>0:
\int_{\C^n}\Psi_q\!\left(\frac{|u(z)|}{\lambda}\right)\dd v(z)
\le1\right\}.
\end{align*}
For bounded sequences, this gives the exact set identity
\begin{equation*}
b\in\ell^{q^-}
\quad\Longleftrightarrow\quad
\sum_j |b_j|^q\bigl(1+\log^+(|b_j|^{-1})\bigr)<\infty,
\end{equation*}
where the summand at $b_j=0$ is zero.
Moreover,
\[
\Psi_q(2t)\le C_q\Psi_q(t),
\qquad 0<t<t_q,
\]
for a fixed $t_q>0$.  Here $c_{00}$ is the set of finitely supported
scalar sequences.  Hence $c_{00}$ is dense in $\ell^{q^-}$.

Set
\[
\mathcal E
=
\{(p,q,r):1<p<2<q<\infty,\ r>\max\{p',q\}\}.
\]
This is the former exceptional range in Garling's classification.  For
$(p,q,r)\in\mathcal E$, define
\[
\vartheta=\frac{1/q-1/r}{1/2-1/r},
\qquad
\frac1\kappa=\frac{1-\vartheta}{r}+\frac{\vartheta}{p'}.
\]
Equivalently,
\[
\kappa
=\frac{p'q(r-2)}{(p'-2)(q-2)+2(r-2)}.
\]
For all $1\le p,q\le\infty$ and $1\le r<\infty$, define
\begin{equation}\label{eq:kappa-full}
\kappa(p,q,r)=
\begin{cases}
\left(\dfrac1{p'}+\dfrac1q-\dfrac12\right)^{-1},
&1\le p\le2,\ 1\le q\le2,\\[6pt]
\infty,&p=1,\ 2<q\le\infty,\\
q,&2<p<\infty,\ 1\le q<p',\\
q^{-},&2<p<\infty,\ 1<q=p'<2,\ 1\le r<q,\\
q,&2<p\le\infty,\ 1\le q=p'<2,\ q\le r<\infty,\\
\max\{p',\min\{r,q\}\},&2\le p\le\infty,\ p'<q\le\infty,\\
p',&1<p<2<q\le\infty,\ 1\le r\le p',\\
r,&1<p<2,\ p'<q\le\infty,\ p'<r\le q,\\[2pt]
\dfrac{p'q(r-2)}{(p'-2)(q-2)+2(r-2)},
&1<p<2<q<\infty,\ r>\max\{p',q\}.
\end{cases}
\end{equation}
The only logarithmic transition is
\begin{equation*}
2<p<\infty,\qquad q=p'<2,\qquad
\mathfrak d_{p,q}^{\,r}
=
\begin{cases}
\ell^{q^-},&1\le r<q,\\
\ell^q,&q\le r<\infty.
\end{cases}
\end{equation*}
Thus the logarithmic factor disappears at the critical index $r=q$.
For equal exponents, \eqref{eq:kappa-full} reduces to
\begin{equation}\label{eq:kappa-diagonal-intro}
\kappa(p,p,r)=
\begin{cases}
2,&1\le p\le2,\\
p',&2\le p<\infty,\ 1\le r\le p',\\
r,&2\le p<\infty,\ p'\le r\le p,\\
p,&2\le p<\infty,\ p\le r<\infty.
\end{cases}
\end{equation}
The value $\infty$ is used when the denominator in the first line is zero.

Define
\[
\mathfrak e_\kappa=
\begin{cases}
\ell^\kappa,&\kappa\in[1,\infty),\\
\ell^\infty,&\kappa=\infty,\\
\ell^{q^-},&\kappa=q^-,
\end{cases}
\qquad
\mathfrak E_\kappa=
\begin{cases}
L^\kappa(\C^n,v),&\kappa\in[1,\infty),\\
L^\infty(\C^n,v),&\kappa=\infty,\\
L^{q^-}(\C^n,v),&\kappa=q^-.
\end{cases}
\]

Recall that a Banach sequence lattice has a solid norm: if
$|c_j|\le|b_j|$ and $b$ is in the space, then $c$ is in the space and
$\|c\|\le\|b\|$.  It is symmetric if coordinate permutations preserve
the norm.

\begin{proposition}\label{thm:diagonal}
For $1\le p,q\le\infty$ and $1\le r<\infty$, the space
$\mathfrak d_{p,q}^{\,r}$ is a symmetric Banach sequence lattice and
\[
D_b\in\Piabs_r(\ell^p,\ell^q)
\quad\Longleftrightarrow\quad
b\in\mathfrak d_{p,q}^{\,r}.
\]
\begin{equation}\label{eq:diagonal-norm}
\pi_r(D_b:\ell^p\to\ell^q)=\|b\|_{\mathfrak d_{p,q}^{\,r}}.
\end{equation}
For every parameter triple,
\begin{equation}\label{eq:garling-identification}
\mathfrak d_{p,q}^{\,r}=\mathfrak e_{\kappa(p,q,r)}
\end{equation}
with equivalence of norms.
\end{proposition}

\begin{proof}
The first assertion and \eqref{eq:diagonal-norm} follow from the definition.
If $|c_j|\le |b_j|$, then $D_c=D_\alpha D_b$ for a diagonal contraction
$D_\alpha$ on $\ell^q$.  The ideal property therefore gives solidity.
Testing on unit vectors gives
\begin{equation}\label{eq:diagonal-linfty}
\|b\|_{\ell^\infty}
\le\|D_b\|
\le\pi_r(D_b)
=\|b\|_{\mathfrak d_{p,q}^{\,r}}.
\end{equation}
Coordinate permutations give symmetry.  To verify completeness, let
$b^{(m)}$ be Cauchy in $\mathfrak d_{p,q}^{\,r}$.  By
\eqref{eq:diagonal-linfty}, it converges coordinatewise to a sequence $b$,
and $D_{b^{(m)}}$ converges in operator norm to $D_b$.  Since
$\Piabs_r(\ell^p,\ell^q)$ is complete, $D_b\in\Piabs_r$ and
$\pi_r(D_{b^{(m)}-b})\to0$.  Garling's theorem
\cite[Theorem~9]{Garling1974} proves
\eqref{eq:garling-identification} outside $\mathcal E$; the remaining case
is \cite[Theorem~1.4]{FanHeWangZeng2026}.  In each case the equivalence of
norms follows from the closed graph theorem, using
\eqref{eq:diagonal-linfty} to obtain coordinatewise convergence.
\end{proof}

The target exponent $q=1$ has no exceptional range and will be used in the
endpoint theorem.

\begin{proposition}
\label{prop:q-one-diagonal}
Let $1\le p,r<\infty$, and let $s_p$ be given by
\eqref{eq:q-one-exponent-intro}.  Then
\begin{equation}\label{eq:q-one-diagonal}
D_b\in\Piabs_r(\ell^p,\ell^1)
\quad\Longleftrightarrow\quad
b\in\ell^{s_p}.
\end{equation}
Moreover,
\begin{equation}\label{eq:q-one-diagonal-norm}
\pi_r(D_b:\ell^p\to\ell^1)
\asymp \|b\|_{\ell^{s_p}}.
\end{equation}
The exponent $s_p$ is independent of $r$.
\end{proposition}

\begin{proof}
Garling's domain is $\ell^{P'}$, so $P=p'$.  Cases (iv)--(v) of
\cite[Theorem~9]{Garling1974} give $s_p=(1/p'+1/2)^{-1}$ for
$p\le2$ and $s_p=1$ for $p>2$, independently of $r$.  This proves the set
identity in \eqref{eq:q-one-diagonal}.  The identity map between
$\ell^{s_p}$ and $\mathfrak d_{p,1}^{\,r}$ is bijective and has closed
graph: convergence in either norm implies coordinatewise convergence, in
the diagonal norm by \eqref{eq:diagonal-linfty}.  The closed graph theorem
and its inverse therefore give the two-sided norm estimate.
\end{proof}

We next isolate the two lattice estimates used in the discretization and
compactness arguments.

\begin{lemma}\label{lem:sequence-tail}
Let $\mathfrak d_{p,q}^{\,r}$ be defined by
\eqref{eq:intrinsic-diagonal-ideal}.
\begin{enumerate}
\item If $c_{00}$ is dense in $\mathfrak d_{p,q}^{\,r}$ and
$b\in\mathfrak d_{p,q}^{\,r}$, then
\begin{equation}\label{eq:tail-vanishing}
\lim_{N\to\infty}
\|b\mathbf1_{\{j>N\}}\|_{\mathfrak d_{p,q}^{\,r}}=0.
\end{equation}
\item Let $\mathcal N(j)$ be finite sets satisfying
\[
\sup_j\#\mathcal N(j)\le N_0,
\qquad
\sup_k\#\{j:k\in\mathcal N(j)\}\le N_0.
\]
Then
\begin{equation}\label{eq:finite-neighbor}
e_j=\sum_{k\in\mathcal N(j)}|b_k|
\quad\Longrightarrow\quad
\|e\|_{\mathfrak d_{p,q}^{\,r}}
\le N_0^2\|b\|_{\mathfrak d_{p,q}^{\,r}}.
\end{equation}
\end{enumerate}
\end{lemma}

\begin{proof}
For the first assertion, choose $c^{(m)}\in c_{00}$ with
\[
\|b-c^{(m)}\|_{\mathfrak d_{p,q}^{\,r}}\longrightarrow0.
\]
If $N$ contains the support of $c^{(m)}$, solidity gives
\[
\|b\mathbf1_{\{j>N\}}\|_{\mathfrak d_{p,q}^{\,r}}
\le
\|b-c^{(m)}\|_{\mathfrak d_{p,q}^{\,r}}.
\]
This proves \eqref{eq:tail-vanishing}.

For each $j$, enumerate $\mathcal N(j)$ as
\[
\mathcal N(j)=\{\sigma_\nu(j):1\le\nu\le m_j\},
\qquad m_j\le N_0.
\]
For fixed $\nu$, every fiber of $\sigma_\nu$ contains at most $N_0$
indices.  Split its domain into $N_0$ sets on which $\sigma_\nu$ is
injective.  This produces at most $N_0^2$ partial injections
$\sigma_{\nu,\mu}$.  Hence
\[
e_j\le
\sum_{\nu=1}^{N_0}\sum_{\mu=1}^{N_0}
|b_{\sigma_{\nu,\mu}(j)}|.
\]
Missing coordinates are read as zero.  For a partial injection
$\sigma:A\to\N$, put
\[
b_j^\sigma=
\begin{cases}
b_{\sigma(j)},&j\in A,\\
0,&j\notin A.
\end{cases}
\]
Define coordinate contractions
\begin{align*}
J_\sigma:\ell^p&\longrightarrow\ell^p,
&(J_\sigma x)_{\sigma(j)}&=x_j\quad(j\in A),\\
Q_\sigma:\ell^q&\longrightarrow\ell^q,
&(Q_\sigma y)_j&=
\begin{cases}
y_{\sigma(j)},&j\in A,\\
0,&j\notin A.
\end{cases}
\end{align*}
All unspecified coordinates of $J_\sigma x$ are zero.  Then
\[
D_{b^\sigma}=Q_\sigma D_bJ_\sigma,
\qquad
\|J_\sigma\|\le1,
\qquad
\|Q_\sigma\|\le1.
\]
The ideal property gives
\[
\|b^\sigma\|_{\mathfrak d_{p,q}^{\,r}}
\le\|b\|_{\mathfrak d_{p,q}^{\,r}}.
\]
Solidity and the triangle inequality now give
\[
\|e\|_{\mathfrak d_{p,q}^{\,r}}
\le
\sum_{\nu,\mu=1}^{N_0}
\|b^{\sigma_{\nu,\mu}}\|_{\mathfrak d_{p,q}^{\,r}}
\le N_0^2\|b\|_{\mathfrak d_{p,q}^{\,r}}.
\]
\end{proof}

\subsection{Lattices and discretization}

A sequence $\Lambda=\{a_j\}$ is a $\delta$-lattice if
$\{B(a_j,\delta)\}$ covers $\C^n$ and the balls $B(a_j,c\delta)$ are
pairwise disjoint for one fixed $c\in(0,1)$.  Every fixed enlargement has
finite overlap.  This means that the number of enlarged balls containing
any one point is bounded uniformly.  A sequence is uniformly separated if
$\inf_{j\ne k}|a_j-a_k|>0$.  Below, \emph{sufficiently separated} means that
this lower bound is chosen large enough for the stated disjoint-ball and
kernel-synthesis estimates.  Every lattice is a finite union of such
sufficiently separated subsequences.

Fix the reference lattice $\Lambda_0=\mathbb Z^{2n}$ and the partition
$\{E_j\}_{j\in\mathbb Z^{2n}}$ defined by
\begin{equation}\label{eq:reference-partition}
E_j=j+[-1/2,1/2)^{2n},
\qquad
B(j,1/2)\subset E_j\subset B(j,\sqrt{2n}/2),
\qquad
\C^n=\bigsqcup_jE_j.
\end{equation}
We fix an enumeration of $\mathbb Z^{2n}$ by $\N$.  Since
$\mathfrak d_{p,q}^{\,r}$ is symmetric, the resulting norm is independent
of this enumeration.
For a measurable function $u$, define
\begin{equation}\label{eq:intrinsic-function-space}
u_j^\#=\operatorname*{ess\,sup}_{z\in E_j}|u(z)|,
\qquad
\|u\|_{\mathfrak D_{p,q}^{\,r}}
=\|\{u_j^\#\}\|_{\mathfrak d_{p,q}^{\,r}}.
\end{equation}
The resulting space is denoted by $\mathfrak D_{p,q}^{\,r}$.  For two fixed
lattice partitions, each cell meets only finitely many cells of the other
partition.  Both multiplicities are uniform.  Hence
\cref{lem:sequence-tail} gives equivalent norms.  Thus
\eqref{eq:intrinsic-function-space} is intrinsic up to lattice constants.
In this paper, a lattice-amalgam space always means this cellwise
essential-supremum construction.

The next lemma collects all radius, lattice, and convolution comparisons used below.

\begin{lemma}\label{lem:discretization}
Let $1\le p,q,r<\infty$.  Let $\Lambda=\{a_j\}$ be a
$\delta$-lattice and let $0<\rho<R<\infty$.  For a positive locally
finite measure $\mu$ and $f\in L^q_{\loc}$, consider
\[
u_{\mu,t}(z)=\mu(B(z,t))^{1/q},\qquad
u_{f,t}(z)=M_{q,t}f(z),\qquad
v_{f,t}(z)=G_{q,t}f(z).
\]
After increasing $R$ by a fixed factor if necessary, each of the three families satisfies
\begin{equation}\label{eq:discrete-continuous}
\|u_\rho\|_{\Xk}\asymp
\|\{u_R(a_j)\}\|_{\xk}\asymp
\|u_R\|_{\Xk}.
\end{equation}
Consequently, membership is independent of the chosen fixed positive radius,
and the corresponding norms are equivalent.  The equivalence constants may
depend on the two radii.  They are uniform over translates of a fixed lattice.

Moreover, suppose that $\mathcal L\ge0$ has an integrable lattice
majorant; that is,
\begin{equation*}
\lambda_m=\sup\{\mathcal L(z-w):z\in E_0,\ w\in E_m\},
\qquad
\{\lambda_m\}\in\ell^1.
\end{equation*}
Here
\[
(\mathcal L*u)(z)=\int_{\C^n}\mathcal L(z-w)u(w)\dd v(w)
\]
denotes convolution.  Then
\begin{equation}\label{eq:convolution-Xk}
\|\mathcal L*u\|_{\Xk}
\lesssim \|\{\lambda_m\}\|_{\ell^1}\|u\|_{\Xk},
\qquad u\in\Xk, u\ge0.
\end{equation}
In particular, \eqref{eq:convolution-Xk} holds for normalized ball kernels
and for $\mathcal L(z)=e^{-c|z|}$.
\end{lemma}

\begin{proof}
\emph{Radius and lattice comparisons.}
Choose a Voronoi-type partition $\{E_j\}$ such that
\[
B(a_j,c_0\delta)\subset E_j\subset B(a_j,\delta),
\qquad v(E_j)\asymp1.
\]
If $z\in E_j$, the inclusions between the relevant balls give, with radii enlarged by at most $2\delta$,
\begin{equation}\label{eq:radius-pointwise}
u_\rho(z)\lesssim u_R(a_j),
\qquad
u_\rho(a_j)\lesssim u_R(z).
\end{equation}
For $u_{\mu,t}$ and $u_{f,t}$, \eqref{eq:radius-pointwise} follows from ball
inclusion.  For $v_{f,t}$, restrict an almost best approximant from the
larger ball.  Fixed ball volumes change only the constant.

By \eqref{eq:intrinsic-function-space}, the first inequality in
\eqref{eq:radius-pointwise} gives
\begin{equation}\label{eq:intrinsic-radius-one}
\|u_\rho\|_{\Xk}
\lesssim
\|\{u_R(a_j)\}\|_{\xk}.
\end{equation}
For $u_{\mu,t}$ and $u_{f,t}$, the reverse comparison follows by covering
every $R$-ball with finitely many $\rho$-balls:
\begin{equation}\label{eq:finite-ball-cover}
u_R(a_j)
\lesssim
\sum_{k\in\mathcal N(j)}u_\rho(a_k),
\qquad
\sup_j\#\mathcal N(j)+\sup_k\#\{j:k\in\mathcal N(j)\}<\infty.
\end{equation}

We give the corresponding argument for $v_{f,t}$.  Fix
$0<t<\rho/8$.  The local construction in
\cite[Lemma~3.6]{HuVirtanen2023} gives a decomposition
$f=F_1+F_2$ such that
\begin{equation}\label{eq:IDA-small-scale-decomposition}
|\bar\partial F_1(w)|+M_{q,t}F_2(w)
\le C G_{q,\rho}f(w),
\qquad w\in\C^n.
\end{equation}
Here the restriction from radius $\rho$ to the radius used in that
construction changes only the fixed constant.

On $B(z,2R)$, let $S_z$ be the local solution of
$\bar\partial S_z=\bar\partial F_1$.  The estimate
\cite[(3.16)]{HuVirtanen2023}, including $q=1$, gives
\begin{equation*}
\|S_z\|_{L^q(B(z,2R))}
\le C_R\|\bar\partial F_1\|_{L^q(B(z,2R))},
\end{equation*}
where $C_R$ is independent of $z$ and $f$.  Since
$F_1-S_z$ is holomorphic on $B(z,2R)$,
\begin{align}
G_{q,R}f(z)
&\le M_{q,R}(F_2+S_z)(z)\notag\\
&\le C_R\left(
M_{q,2R}F_2(z)+M_{q,2R}(|\bar\partial F_1|)(z)
\right).
\label{eq:IDA-large-from-small}
\end{align}

Let $c_k=\operatorname*{ess\,sup}_{E_k}G_{q,\rho}f$.  Covering the
fixed ball in \eqref{eq:IDA-large-from-small} by finitely many $t$-balls
and using \eqref{eq:IDA-small-scale-decomposition} gives
\begin{equation}\label{eq:IDA-finite-neighbor}
\operatorname*{ess\,sup}_{z\in E_j}G_{q,R}f(z)
\le C\sum_{k\in\mathcal N(j)}c_k,
\end{equation}
where
\[
\sup_j\#\mathcal N(j)
+\sup_k\#\{j:k\in\mathcal N(j)\}<\infty.
\]
The two multiplicities depend only on $t,R$, and the fixed lattice
geometry.  Therefore \cref{lem:sequence-tail} gives
\[
\|G_{q,R}f\|_{\Xk}\lesssim\|G_{q,\rho}f\|_{\Xk}.
\]
The opposite inequality follows from restriction of an almost best
approximant on $B(z,R)$ to $B(z,\rho)$.  Finally, solidity,
\eqref{eq:finite-neighbor}, and \eqref{eq:intrinsic-radius-one} prove
\eqref{eq:discrete-continuous} for all three families.

\emph{Convolution estimate.}
For the convolution estimate, put $c_m=u_m^\#$ as in
\eqref{eq:intrinsic-function-space}.  If $z\in E_j$, then
\begin{align*}
(\mathcal L*u)(z)
&=\sum_m\int_{E_m}\mathcal L(z-w)u(w)\dd v(w)\\
&\le C\sum_m \lambda_{m-j}c_m.
\end{align*}
For a sequence $c=(c_j)$, let $\tau_\ell c$ denote its coordinate
translation, $(\tau_\ell c)_j=c_{j+\ell}$.  Symmetry of the diagonal ideal
and the triangle inequality yield
\begin{align*}
\|\mathcal L*u\|_{\Xk}
&\lesssim
\left\|\sum_\ell \lambda_\ell\tau_\ell c\right\|_{\xk}\\
&\le
\sum_\ell \lambda_\ell\|\tau_\ell c\|_{\xk}
=\|\lambda\|_{\ell^1}\|u\|_{\Xk}.
\end{align*}
This proves \eqref{eq:convolution-Xk}.  The stated kernels have exponentially decreasing, respectively finitely supported, lattice majorants.
\end{proof}

The comparison lemma converts the intrinsic lattice norm into an explicit
Lebesgue, supremum, or Orlicz norm.

\begin{proposition}
\label{prop:explicit-realization}
Let $1\le p,q,r<\infty$, and let
$\kappa=\kappa(p,q,r)$ be given by \eqref{eq:kappa-full}.  For each of the local quantities
\[
u_{\mu,t},\qquad u_{f,t},\qquad v_{f,t}
\]
in \cref{lem:discretization},
\begin{equation}\label{eq:explicit-realization}
\|u\|_{\Xk}\asymp\|u\|_{\mathfrak E_\kappa}.
\end{equation}
Thus the intrinsic condition becomes a Lebesgue, supremum, or logarithmic
Orlicz condition.  The relevant case is given in \eqref{eq:kappa-full}.
\end{proposition}

\begin{proof}
Fix the radius in $u$.  Choose the smaller and larger radii used in
\cref{lem:discretization}.  Let $\{E_j\}$ be the corresponding lattice
partition, and put $c_j=u_R(a_j)$.  Then
\begin{equation}\label{eq:explicit-discrete-chain}
\|u\|_{\Xk}
\asymp\|c\|_{\mathfrak d_{p,q}^{\,r}}
\asymp\|c\|_{\mathfrak e_\kappa}.
\end{equation}

We next compare $u$ with the step function
$u_c=\sum_jc_j\mathbf1_{E_j}$.  The pointwise estimates
\eqref{eq:radius-pointwise}, \eqref{eq:finite-ball-cover}, and
\eqref{eq:IDA-finite-neighbor} give both comparisons after a fixed change
of radius.  Each comparison uses only finite-neighbor sums.  Both neighbor
multiplicities are bounded.  Thus, if $\kappa=s\in[1,\infty)$,
\[
\|u\|_{L^s}^s\asymp\|u_c\|_{L^s}^s
\asymp\sum_j|c_j|^s.
\]
Here we used $v(E_j)\asymp1$ and finite overlap.  The same pointwise
comparisons give $\|u\|_\infty\asymp\|c\|_{\ell^\infty}$.

Suppose now that $\kappa=q^-$.  For every $\lambda>0$,
\begin{equation*}
\int_{\C^n}\Psi_q\!\left(
\frac{u_c(z)}{\lambda}
\right)\dd v(z)
=\sum_jv(E_j)\Psi_q\!\left(\frac{c_j}{\lambda}\right).
\end{equation*}
Convexity of $\Psi_q$ and the bounded neighbor multiplicities show that
the finite-neighbor comparisons are bounded in the Luxemburg norm.
Indeed, if
$a_j\le C\sum_{k\in\mathcal N(j)}c_k$, with at most $N_0$ terms in
each sum, then convexity gives
\[
\Psi_q\!\left(\frac{a_j}{CN_0\lambda}\right)
\le \frac1{N_0}\sum_{k\in\mathcal N(j)}
\Psi_q\!\left(\frac{c_k}{\lambda}\right).
\]
After summing in $j$, the reverse multiplicity bound shows that the right
side is at most a fixed multiple of
$\sum_k\Psi_q(c_k/\lambda)$.  Rescaling $\lambda$ yields the Luxemburg
norm estimate.
Hence $\|u\|_{L^{q^-}}\asymp\|c\|_{\ell^{q^-}}$.
Combining these estimates with \eqref{eq:explicit-discrete-chain} proves
\eqref{eq:explicit-realization}.
\end{proof}

For $f\in L^q_{\loc}$, define
\begin{equation*}
\begin{split}
G_{q,\delta}f(z)
&=\inf_{h\in\Hol(B(z,\delta))}
\left(\avg_{B(z,\delta)}|f-h|^q\dd v\right)^{1/q},\\
M_{q,\delta}f(z)
&=\left(\avg_{B(z,\delta)}|f|^q\dd v\right)^{1/q},\\
a_\delta f(z)&=\avg_{B(z,\delta)}f\dd v.
\end{split}
\end{equation*}
Here $G_{q,\delta}f$ is the normalized local $L^q$-distance from $f$ to
holomorphic functions.  It measures the nonanalytic part of $f$.
The function $M_{q,\delta}f$ is the normalized local $q$-mean and measures
the local size of $f$.  The function $a_\delta f$ is the complex ball
average.  Unlike $M_{q,\delta}f$, it retains phase information and therefore
records cancellation.  These are the three local quantities introduced in
the Introduction.
Thus \eqref{eq:IDA-class-intro} is independent of the fixed radius.  Notice that
\begin{equation}\label{eq:G-less-M}
G_{q,\delta}f\le M_{q,\delta}f.
\end{equation}

\section{Absolutely summing Fock--Carleson embeddings}\label{sec:carleson}

Let $\mu$ be a positive locally finite Borel measure on $\C^n$.  Define
\[
L_\varphi^q(\mu)=L^q(\C^n,e^{-q\varphi}\dd\mu)
\]
and let
\[
J_\mu:F_\varphi^p\longrightarrow L_\varphi^q(\mu),
\qquad J_\mu g=g,
\]
whenever this inclusion is well defined.  For a fixed radius $\delta>0$, set
\begin{equation}\label{eq:mu-density}
\widehat\mu_{q,\delta}(z)=\mu(B(z,\delta))^{1/q}.
\end{equation}
The Euclidean volume of a fixed-radius ball is constant.  Thus no boundary factor occurs in \eqref{eq:mu-density}, in contrast to the Bergman setting.

We first state the block-diagonal principle needed below.  The same abstract
principle appears as Proposition~4.2 in \cite{FanHeWangZeng2026}.  We retain
the proof to specify the normalization used in the Fock localization.

For Banach spaces $X_j$ and $1\le s<\infty$, the direct sum
$(\bigoplus_jX_j)_{\ell^s}$ consists of all sequences $x=(x_j)$ with
$x_j\in X_j$ and
\[
\|x\|=\left(\sum_j\|x_j\|_{X_j}^s\right)^{1/s}<\infty.
\]
For $s=\infty$, the norm is $\sup_j\|x_j\|_{X_j}$.

\begin{proposition}\label{prop:block-suff}
Let $1\le p\le\infty$, $1\le q,r<\infty$.  Set
$X=(\bigoplus_jX_j)_{\ell^p}$ and $Y=(\bigoplus_jY_j)_{\ell^q}$.
Suppose that $A_j:X_j\to Y_j$ is $1$-summing.  Put
$d_j=\pi_1(A_j)$.  If $D_d:\ell^p\to\ell^q$ is $r$-summing, then
\[
A=\bigoplus_jA_j:X\longrightarrow Y
\]
is $r$-summing and
\begin{equation*}
\pi_r(A)\le\pi_r(D_d).
\end{equation*}
\end{proposition}

\begin{proof}
Write $d_j=\pi_1(A_j)$.  The boundedness of $D_d$ first gives
\[
\left(\sum_j\|A_jx_j\|_{Y_j}^q\right)^{1/q}
\le
\left(\sum_jd_j^q\|x_j\|_{X_j}^q\right)^{1/q}
\le \|D_d\|\,\|x\|_X.
\]
Thus the block-diagonal map is well defined.  By Pietsch domination
\cite{Pietsch1967,DiestelJarchowTonge1995}, for each $j$ there is a
probability measure $\nu_j$ on $B_{X_j^*}$ such that
\[
\|A_jx\|_{Y_j}
\le d_j\int_{B_{X_j^*}}|x^*(x)|\dd\nu_j(x^*).
\]

Fix $N$ and put
\[
(\Omega_N,\nu^{(N)})=\prod_{j=1}^N(B_{X_j^*},\nu_j).
\]
For $\omega=(x_1^*,\ldots,x_N^*)\in\Omega_N$, define
\[
C_{\omega,N}(x_j)=
(x_1^*(x_1),\ldots,x_N^*(x_N),0,\ldots).
\]
Then $\|C_{\omega,N}:X\to\ell^p\|\le1$.  If $A^{(N)}$ denotes the truncation to the first $N$ blocks, Minkowski's inequality yields
\begin{equation*}
\|A^{(N)}x\|_Y
\le\int_{\Omega_N}\|D_dC_{\omega,N}x\|_{\ell^q}\dd\nu^{(N)}(\omega).
\end{equation*}
For $x^{(1)},\ldots,x^{(m)}\in X$, apply Minkowski's inequality once more and then the $r$-summing inequality for $D_d$:
\begin{align*}
\left(\sum_{k=1}^m\|A^{(N)}x^{(k)}\|_Y^r\right)^{1/r}
&\le\int_{\Omega_N}
\left(\sum_{k=1}^m
\|D_dC_{\omega,N}x^{(k)}\|_{\ell^q}^r\right)^{1/r}
\dd\nu^{(N)}(\omega)\\
&\le\pi_r(D_d)
\sup_{x^*\in B_{X^*}}
\left(\sum_{k=1}^m|x^*(x^{(k)})|^r\right)^{1/r}.
\end{align*}
Indeed, every $\xi\in B_{(\ell^p)^*}$ gives
$\xi\circ C_{\omega,N}\in B_{X^*}$.  Hence
\[
\pi_r(A^{(N)})\le\pi_r(D_d),
\qquad N\ge1.
\]
For a fixed $x\in X$, the vectors $A^{(N)}x$ converge to $Ax$ in $Y$.
Indeed, their $q$th-power tails are dominated by the corresponding tails
of $(d_j\|x_j\|_{X_j})\in\ell^q$.  Apply this convergence to each member
of a fixed finite family.  Letting $N\to\infty$ in the finite-family
inequality proves the assertion.
\end{proof}

For a Euclidean ball $U$, write
\[
A^p(U,e^{-p\varphi}\dd v)
=L^p(U,e^{-p\varphi}\dd v)\cap\Hol(U).
\]

The local building block is the following restriction estimate.

\begin{lemma}\label{lem:local-restriction}
Let $1\le p,q<\infty$.  Fix $\delta>0$ and a positive Borel measure
$\mu$.  For $a\in\C^n$, define
\[
R_a:A^p(B(a,3\delta),e^{-p\varphi}\dd v)
\longrightarrow L^q(B(a,\delta),e^{-q\varphi}\dd\mu),
\qquad R_ag=g|_{B(a,\delta)}.
\]
Then $R_a$ is $1$-summing if and only if $\mu(B(a,\delta))<\infty$, and
\begin{equation}\label{eq:local-restriction-norm}
\pi_1(R_a)\asymp\mu(B(a,\delta))^{1/q}.
\end{equation}
The constants are independent of $a$ and $\mu$.
\end{lemma}

\begin{proof}
Apply the gauge in \cref{lem:gauge}.  Multiply by $e^{-L_a}$ and cancel
the factor $e^{-\varphi(a)}$.  It is then enough to use unweighted
holomorphic functions on concentric balls of fixed radii.

Let $\nu_a$ be normalized volume measure on $B(a,2\delta)$.  The restriction factors as
\[
A^p(B(a,3\delta))
\xrightarrow{U_a}
H^\infty(B(a,2\delta))
\xrightarrow{I_a}
A^1(B(a,2\delta),\nu_a)
\xrightarrow{V_a}
L^q(B(a,\delta),\mu).
\]
Thus
$\nu_a(E)=v(E\cap B(a,2\delta))/v(B(a,2\delta))$.
Here $H^\infty(U)$ denotes the bounded holomorphic functions on $U$, and
$A^1(U,\nu)=L^1(U,\nu)\cap\Hol(U)$.
The holomorphic submean inequality gives $\|U_a\|\le C$.  To estimate
$I_a$, let $h_1,\ldots,h_m\in H^\infty(B(a,2\delta))$.  Every point
evaluation has norm at most one.  Hence
\begin{align*}
\sum_{k=1}^m\|I_ah_k\|_{A^1(\nu_a)}
&=\int_{B(a,2\delta)}\sum_{k=1}^m|h_k(w)|\dd\nu_a(w)\\
&\le\sup_{x^*\in B_{(H^\infty)^*}}
\sum_{k=1}^m|x^*(h_k)|.
\end{align*}
Thus $\pi_1(I_a)\le1$.  A second submean estimate gives
\[
\|V_a h\|_{L^q(B(a,\delta),\mu)}
\le C\mu(B(a,\delta))^{1/q}\|h\|_{A^1(B(a,2\delta),\nu_a)}.
\]
The constants are uniform in $a$.  The ideal property proves the upper estimate in \eqref{eq:local-restriction-norm}.

For the reverse estimate, use the gauge function $e^{L_a}$.  By
\eqref{eq:gauge-norm},
\begin{align*}
\|e^{L_a}\|_{A^p(B(a,3\delta),e^{-p\varphi})}
&\asymp e^{-\varphi(a)},\\
\|e^{L_a}\|_{L^q(B(a,\delta),e^{-q\varphi}\dd\mu)}
&\asymp e^{-\varphi(a)}\mu(B(a,\delta))^{1/q}.
\end{align*}
Hence
\[
\pi_1(R_a)\ge\|R_a\|\gtrsim\mu(B(a,\delta))^{1/q}.
\]
\end{proof}

We also need a necessity principle.  Partition a lattice into finitely
many sufficiently separated families.  A standard atomic estimate under
\eqref{eq:hessian} says that the synthesis map
\begin{equation*}
S_\Lambda(c_j)=\sum_jc_jk_{a_j}
\end{equation*}
is bounded from $\ell^p$ to $F_\varphi^p$; see \cite{HuLv2014,SchusterVarolin2012}.

\begin{proposition}[Rademacher diagonal extraction]\label{prop:diagonal-extraction}
Let $1\le p,q,r<\infty$ and fix $\delta>0$.  Let $\mu$ be a positive
locally finite Borel measure.  Suppose that
$A:F_\varphi^p\to L_\varphi^q(\mu)$ is
$r$-summing and that $\Lambda=\{a_j\}$ is sufficiently separated so that
$B(a_j,\delta)$ are pairwise disjoint and the synthesis map
$S_\Lambda:\ell^p\to F_\varphi^p$ is bounded.  Put
\[
b_j=
\left(
\int_{B(a_j,\delta)}|Ak_{a_j}(z)|^qe^{-q\varphi(z)}\dd\mu(z)
\right)^{1/q}.
\]
Then $D_b:\ell^p\to\ell^q$ is $r$-summing and
\begin{equation*}
\pi_r(D_b)\lesssim\pi_r(A).
\end{equation*}
\end{proposition}

\begin{proof}
Let $Y_j=L^q(B(a_j,\delta),e^{-q\varphi}\dd\mu)$ and normalize
\[
u_j=b_j^{-1}(Ak_{a_j})|_{B(a_j,\delta)}
\]
when $b_j>0$; indices for which $b_j=0$ may be omitted.  Thus
$\|u_j\|_{Y_j}=1$.  Let $\{\varepsilon_j\}$ be the Rademacher system on
$[0,1]$.  These are independent functions with values in $\{-1,1\}$ and
\[
\int_0^1\varepsilon_i(t)\varepsilon_j(t)\dd t
=\begin{cases}1,&i=j,\\0,&i\ne j.\end{cases}
\]
For finitely supported $c=(c_j)$, define
\[
\Theta_t c=\sum_j\varepsilon_j(t)c_jk_{a_j}
\]
and
\[
R_tg=\{\varepsilon_j(t)g|_{B(a_j,\delta)}\}_j.
\]
The synthesis estimate gives $\sup_t\|\Theta_t\|<\infty$.  Separation gives $\sup_t\|R_t\|<\infty$.  Therefore
\[
\Delta c=\int_0^1R_tA\Theta_tc\dd t
\]
defines an $r$-summing map from $\ell^p$ to $(\bigoplus_jY_j)_{\ell^q}$ with
\[
\pi_r(\Delta)\lesssim\pi_r(A).
\]
The map $t\mapsto R_tA\Theta_t$ is strongly measurable on every finite
coordinate truncation.  The ideal property and the triangle inequality
give
\[
\pi_r(\Delta)
\le\int_0^1\|R_t\|\,\pi_r(A)\,\|\Theta_t\|\dd t
\lesssim\pi_r(A).
\]
For $c\in c_{00}$, the $j$th coordinate of the integrand equals
\[
\varepsilon_j(t)\sum_i\varepsilon_i(t)c_i
(Ak_{a_i})|_{B(a_j,\delta)}.
\]
Rademacher orthogonality therefore gives
\[
\Delta c=\{b_jc_ju_j\}_j.
\]
The map $\Phi:\ell^q\to(\bigoplus_jY_j)_{\ell^q}$ defined by
$\Phi((d_j))=(d_ju_j)$ is an isometry onto a closed subspace.  The set
$c_{00}$ is dense in $\ell^p$.  Thus the identity extends to all
$c\in\ell^p$, and $\Delta(\ell^p)\subset\Phi(\ell^q)$.  Consequently,
\[
D_b=\Phi^{-1}\Delta,
\qquad
\pi_r(D_b)\le\pi_r(\Delta)\lesssim\pi_r(A).
\]
\end{proof}

The block estimate gives sufficiency.  Diagonal extraction gives necessity.
Together they yield the first principal theorem.

\begin{theorem}[Fock--Carleson embeddings]\label{thm:carleson}
Let $1\le p,q,r<\infty$.  Let $\mu$ be a positive locally finite Borel
measure on $\C^n$.  Then the following conditions are equivalent:
\begin{enumerate}
\item $J_\mu\in\Piabs_r(F_\varphi^p,L_\varphi^q(\mu))$;
\item $\widehat\mu_{q,\delta}\in\Xk$ for some, equivalently every, $\delta>0$;
\item $\{\mu(B(a_j,\delta))^{1/q}\}_j\in\xk$ for some, equivalently every, $\delta$-lattice.
\end{enumerate}
Moreover,
\begin{equation*}
\pi_r(J_\mu)
\asymp
\|\widehat\mu_{q,\delta}\|_{\Xk}
\asymp
\|\{\mu(B(a_j,\delta))^{1/q}\}\|_{\xk}.
\end{equation*}
\end{theorem}

\begin{proof}
\emph{Sufficiency.}
By \cref{lem:discretization}, fix $\delta_0>0$ so that
\eqref{eq:kernel-lower} holds on $B(a,4\delta_0)$.  Let $\{a_j\}$ be a
$\delta_0$-lattice.  Finite overlap gives a bounded map
\[
U:F_\varphi^p\longrightarrow
\left(\bigoplus_jA^p(B(a_j,3\delta_0),e^{-p\varphi}\dd v)\right)_{\ell^p},
\qquad Ug=\{g|_{B(a_j,3\delta_0)}\},
\]
because
\[
\|Ug\|^p
=\sum_j\int_{B(a_j,3\delta_0)}|g|^pe^{-p\varphi}\dd v
\lesssim\|g\|_{p,\varphi}^p.
\]
Choose a measurable partition $\{E_j\}$ of $\C^n$ with $E_j\subset B(a_j,\delta_0)$.  Restriction to $E_j$ defines a contraction from
$(\bigoplus_jL^q(B(a_j,\delta_0),e^{-q\varphi}\dd\mu))_{\ell^q}$ to $L_\varphi^q(\mu)$.

By \cref{lem:local-restriction}, the local restriction maps $R_{a_j}$ are $1$-summing and
\[
\pi_1(R_{a_j})\asymp d_j,
\qquad d_j=\mu(B(a_j,\delta_0))^{1/q}.
\]
If $d=(d_j)\in\xk$, then \cref{thm:diagonal} gives
$D_d\in\Piabs_r(\ell^p,\ell^q)$.  Apply \cref{prop:block-suff}.  The
localization and gluing maps then show that $J_\mu$ is $r$-summing and
\[
\pi_r(J_\mu)\lesssim\|d\|_{\xk}.
\]

\emph{Necessity.}
Conversely, write the lattice as
$\Lambda=\Lambda_1\cup\cdots\cup\Lambda_N$, where each $\Lambda_\nu$ is
sufficiently separated.  Apply \cref{prop:diagonal-extraction} to
$A=J_\mu$ on each $\Lambda_\nu$.  For $a_j\in\Lambda_\nu$,
\eqref{eq:kernel-lower} gives
\[
\|J_\mu k_{a_j}\|_{L^q(B(a_j,\delta_0),e^{-q\varphi}\dd\mu)}^q
\gtrsim\mu(B(a_j,\delta_0)).
\]
The ideal property for diagonal maps and \cref{thm:diagonal} yield, on each sublattice,
\[
\|\{d_j:a_j\in\Lambda_\nu\}\|_{\xk}
\lesssim\pi_r(J_\mu).
\]
Since $N$ is fixed,
\[
\|d\|_{\xk}
\le\sum_{\nu=1}^N
\|d\mathbf1_{\Lambda_\nu}\|_{\xk}
\lesssim\pi_r(J_\mu).
\]
The argument also applies to $\ell^{q^-}$, which is a Banach lattice.
Finally, \cref{lem:discretization} gives the continuous estimate and radius
independence.
\end{proof}

The same localization argument yields compactness when the diagonal lattice
has order-continuous norm.

\begin{proposition}
\label{prop:carleson-compact}
Under the assumptions of \cref{thm:carleson}, suppose that $c_{00}$ is dense in $\xk$.  If $J_\mu\in\Piabs_r$, then $J_\mu$ is compact.
More precisely, for every $\delta$-lattice $\{a_j\}_{j\ge1}$ and the associated sequence
\[
d_j=\mu(B(a_j,2\delta))^{1/q},
\]
there are compact operators $C_N:F_\varphi^p\to L_\varphi^q(\mu)$ such that
\begin{equation}\label{eq:compact-tail-estimate}
\|J_\mu-C_N\|
\le\pi_r(J_\mu-C_N)
\lesssim
\|\{d_j:j>N\}\|_{\xk}
\longrightarrow0.
\end{equation}
\end{proposition}

\begin{proof}
Choose a measurable partition $\{E_j\}$ satisfying
\[
E_j\subset B(a_j,\delta),
\qquad
\C^n=\bigsqcup_{j\ge1}E_j.
\]
Set
\[
\Omega_N=\bigcup_{j=1}^NE_j,
\qquad
\dd\mu_N=\mathbf1_{\Omega_N}\dd\mu,
\qquad
\dd\nu_N=\mathbf1_{\C^n\setminus\Omega_N}\dd\mu.
\]
Let $C_Ng=\mathbf1_{\Omega_N}g$, regarded as an operator into
$L_\varphi^q(\mu)$.  We first prove that $C_N$ is compact.  If
$\{g_m\}$ is bounded in $F_\varphi^p$, the weighted submean estimate gives
\begin{equation*}
\sup_{z\in B(0,R)}|g_m(z)|e^{-\varphi(z)}\le C_R,
\qquad m\ge1.
\end{equation*}
Montel's theorem yields a subsequence, still denoted by $\{g_m\}$, and an entire function $g$ such that
\[
g_m\longrightarrow g
\quad\text{uniformly on compact subsets of }\C^n.
\]
Fatou's lemma gives
$\|g\|_{p,\varphi}\le\liminf_m\|g_m\|_{p,\varphi}$, so
$g\in F_\varphi^p$.
The set $\Omega_N$ is bounded and $\mu(\Omega_N)<\infty$.  Hence
\begin{align*}
\|C_Ng_m-C_Ng\|_{L_\varphi^q(\mu)}^q
&=\int_{\Omega_N}|g_m-g|^qe^{-q\varphi}\dd\mu\\
&\le \mu(\Omega_N)
\sup_{\Omega_N}\bigl(|g_m-g|^qe^{-q\varphi}\bigr)\longrightarrow0.
\end{align*}
Thus $C_N$ is compact.

It remains to estimate the tail.  Put
\[
e_k^{(N)}=\nu_N(B(a_k,\delta))^{1/q}.
\]
If $E_j\cap B(a_k,\delta)\ne\varnothing$, then
$|a_j-a_k|<2\delta$.  Consequently,
\begin{align}
e_k^{(N)}
&\le
\left(
\sum_{\substack{j>N\\ |a_j-a_k|<2\delta}}
\mu(E_j)
\right)^{1/q}\notag\\
&\le
\sum_{\substack{j>N\\ |a_j-a_k|<2\delta}}
\mu(B(a_j,2\delta))^{1/q}
=
\sum_{\substack{j>N\\ |a_j-a_k|<2\delta}}d_j.
\label{eq:tail-neighbors}
\end{align}
Each sum in \eqref{eq:tail-neighbors} has a bounded number of terms.  A
fixed $d_j$ also occurs in a bounded number of these sums.  Therefore,
\cref{lem:sequence-tail} gives
\begin{equation}\label{eq:tail-xk}
\|e^{(N)}\|_{\xk}
\lesssim
\|\{d_j:j>N\}\|_{\xk}\longrightarrow0.
\end{equation}
Since
\[
J_\mu-C_N=J_{\nu_N},
\]
where the target is identified isometrically with the corresponding closed subspace of $L_\varphi^q(\mu)$, \cref{thm:carleson} and
\eqref{eq:tail-xk} give \eqref{eq:compact-tail-estimate}.
\end{proof}

We now apply \cref{thm:carleson} to the measure
$\dd\mu_f=|f|^q\dd v$.

\begin{corollary}\label{cor:multiplication}
Let $1\le p,q,r<\infty$, fix $\delta>0$, and let
$f\in L^q_{\loc}(\C^n)$.  Then
\[
M_f:F_\varphi^p\longrightarrow L_\varphi^q
\]
is $r$-summing if and only if $M_{q,\delta}f\in\Xk$.  Moreover,
\begin{equation}\label{eq:multiplication-norm}
\pi_r(M_f)\asymp\|M_{q,\delta}f\|_{\Xk}.
\end{equation}
\end{corollary}

\begin{proof}
Put $\dd\mu_f=|f|^q\dd v$.  The map
\[
U:L_\varphi^q(\mu_f)\to L_\varphi^q,
\qquad Uh=fh,
\]
is an isometry, and $M_f=UJ_{\mu_f}$.  Conversely,
\[
V:L_\varphi^q\to L_\varphi^q(\mu_f),
\qquad
Vh=\frac{h}{f}\,\mathbf1_{\{f\ne0\}},
\]
is a contraction and $J_{\mu_f}=VM_f$.  Hence $\pi_r(M_f)=\pi_r(J_{\mu_f})$.  Since
\[
\mu_f(B(z,\delta))^{1/q}
=|B(0,\delta)|^{1/q}M_{q,\delta}f(z),
\]
the result follows from \cref{thm:carleson}.
\end{proof}

The preceding compactness criterion immediately transfers to multiplication.

\begin{corollary}\label{cor:multiplication-compact}
Let the assumptions of \cref{cor:multiplication} hold and suppose that
$c_{00}$ is dense in $\xk$.  Then
\[
M_f\in\Piabs_r(F_\varphi^p,L_\varphi^q)
\quad\Longrightarrow\quad
M_f\text{ is compact}.
\]
Consequently, $T_f=PM_f$ and $H_f=(I-P)M_f$ are compact.
\end{corollary}

\begin{proof}
For $\dd\mu_f=|f|^q\dd v$, the isometries in the proof of
\cref{cor:multiplication} identify $M_f$ with $J_{\mu_f}$.  Apply
\cref{prop:carleson-compact}.  Moreover,
$M_{q,\delta}f\in\Xk$ by \cref{cor:multiplication}.  By
\cref{lem:discretization}, for a fixed lattice and a fixed $R>0$,
\[
\{M_{q,R}f(a_j)\}_j\in\xk\subset\ell^\infty.
\]
Kernel decay and a lattice partition give, for every $z\in\C^n$,
\[
\int_{\C^n}|f(w)k_z(w)|e^{-\varphi(w)}\dd v(w)
\lesssim
\sum_j e^{-c|z-a_j|}M_{q,R}f(a_j)<\infty.
\]
Thus $f\in\mathcal D_\varphi$.  The operators
$T_f$ and $H_f$ are therefore defined on $\Gamma$, and
\[
T_f=PM_f,
\qquad
H_f=(I-P)M_f
\]
on the core and on their unique bounded extensions.  Their compactness
follows from the boundedness of $P$ and $I-P$ on $L_\varphi^q$.
\end{proof}

\section{IDA decomposition and scalar recovery}\label{sec:hankel}

\subsection{Hankel operators}

We now characterize the nonanalytic part of the symbol.  The next result
is the IDA decomposition from \cite{HuVirtanen2023}, adapted to the
uniform real-Hessian setting.

\begin{lemma}[IDA decomposition]\label{lem:ida-decomposition}
Let $1\le p,q,r<\infty$.  Fix $\delta>0$.  Suppose that
$f\in\mathcal D_\varphi\cap L^q_{\loc}$ and
$G_{q,\delta}f\in\Xk$.  There are
\[
f_1\in C^\infty(\C^n)\cap\mathcal D_\varphi,
\qquad
f_2\in L^q_{\loc}\cap\mathcal D_\varphi
\]
such that $f=f_1+f_2$.  There are also fixed radii $0<\rho<R$ and a $\rho$-lattice $\{a_j\}$ for which
\begin{equation}\label{eq:ida-decomposition}
M_{q,\rho}(|\bar\partial f_1|)(z)
+M_{q,\rho}f_2(z)
\le C\sum_{a_j\in B(z,R)}G_{q,R}f(a_j)
\end{equation}
for every $z\in\C^n$.  Consequently,
\begin{equation}\label{eq:ida-decomposition-norm}
\|M_{q,\delta}(|\bar\partial f_1|)\|_{\Xk}
+\|M_{q,\delta}f_2\|_{\Xk}
\lesssim
\|G_{q,R}f\|_{\Xk}.
\end{equation}
\end{lemma}

\begin{proof}
\emph{Construction.}
Choose a $\rho$-lattice $\{a_j\}$ with $\rho>0$ small.  Let $\{\chi_j\}$ be a $C^\infty$ partition of unity such that
\[
\supp\chi_j\subset B(a_j,2\rho),
\qquad
|\bar\partial\chi_j|\le C\rho^{-1},
\qquad
\sum_j\chi_j=1.
\]
For each $j$, choose $h_j\in\Hol(B(a_j,8\rho))$ satisfying
\begin{equation}\label{eq:almost-best}
\left(\avg_{B(a_j,8\rho)}|f-h_j|^q\dd v\right)^{1/q}
\le2G_{q,8\rho}f(a_j).
\end{equation}
Put
\[
f_1=\sum_j\chi_jh_j,
\qquad
f_2=f-f_1=\sum_j\chi_j(f-h_j).
\]
The sums are locally finite.  Hence $f_1\in C^\infty$ and
$f_2\in L^q_{\loc}$.  Suppose
$B(a_j,2\rho)\cap B(a_k,2\rho)\ne\varnothing$.  A fixed ball centered in
the overlap lies in both enlarged balls.  Thus \eqref{eq:almost-best} and
the holomorphic submean inequality give
\begin{equation}\label{eq:approximant-difference}
\sup_{B(a_j,2\rho)\cap B(a_k,2\rho)}|h_j-h_k|
\lesssim G_{q,8\rho}f(a_j)+G_{q,8\rho}f(a_k).
\end{equation}
Indeed, fix $x$ in the overlap.  Then
$B(x,\rho)\subset B(a_j,8\rho)\cap B(a_k,8\rho)$.  Since $h_j-h_k$ is holomorphic,
\begin{align*}
|h_j(x)-h_k(x)|
&\lesssim
\left(\avg_{B(x,\rho)}|h_j-h_k|^q\dd v\right)^{1/q}\\
&\le
\left(\avg_{B(x,\rho)}|h_j-f|^q\dd v\right)^{1/q}
+\left(\avg_{B(x,\rho)}|f-h_k|^q\dd v\right)^{1/q}\\
&\lesssim G_{q,8\rho}f(a_j)+G_{q,8\rho}f(a_k).
\end{align*}

\emph{Local estimates.}
Fix $z$ and select $j(z)$ with $z\in B(a_{j(z)},\rho)$.  Only a bounded
set $N(z)$ of indices occurs on $B(z,\rho)$.  The bound is uniform in
$z$.  From \eqref{eq:almost-best},
\begin{equation}\label{eq:f2-local}
M_{q,\rho}f_2(z)
\lesssim\sum_{j\in N(z)}G_{q,8\rho}f(a_j).
\end{equation}
Indeed, finite overlap and convexity give
\begin{align*}
\avg_{B(z,\rho)}|f_2|^q\dd v
&=\avg_{B(z,\rho)}
\left|\sum_{j\in N(z)}\chi_j(f-h_j)\right|^q\dd v\\
&\lesssim
\sum_{j\in N(z)}
\avg_{B(a_j,8\rho)}|f-h_j|^q\dd v
\lesssim
\sum_{j\in N(z)}G_{q,8\rho}f(a_j)^q.
\end{align*}
Moreover, $\sum_j\bar\partial\chi_j=0$, and hence
\[
\bar\partial f_1
=\sum_{j\in N(z)}\bar\partial\chi_j(h_j-h_{j(z)})
\quad\hbox{on }B(z,\rho).
\]
Together with \eqref{eq:approximant-difference}, this gives
\begin{equation}\label{eq:df1-local}
M_{q,\rho}(|\bar\partial f_1|)(z)
\lesssim\sum_{j\in N(z)}G_{q,8\rho}f(a_j).
\end{equation}
More precisely, \eqref{eq:approximant-difference} gives, for
$w\in B(z,\rho)$,
\begin{align*}
|\bar\partial f_1(w)|
&\le
\sum_{j\in N(z)}|\bar\partial\chi_j(w)|
|h_j(w)-h_{j(z)}(w)|\\
&\lesssim
\rho^{-1}\sum_{j\in N(z)}
\bigl(G_{q,8\rho}f(a_j)+G_{q,8\rho}f(a_{j(z)})\bigr).
\end{align*}
The fixed factor $\rho^{-1}$ is absorbed in the constant.
Equations \eqref{eq:f2-local} and \eqref{eq:df1-local} prove
\eqref{eq:ida-decomposition}.  On the lattice, both right-hand sides are
finite-neighbor sums.  Hence \cref{lem:sequence-tail,lem:discretization}
give
\[
\|M_{q,\delta}(|\bar\partial f_1|)\|_{\Xk}
+\|M_{q,\delta}f_2\|_{\Xk}
\lesssim\|G_{q,R}f\|_{\Xk}.
\]

\emph{Initial domains.}
It remains to verify the initial domains.  Since
\[
\{G_{q,8\rho}f(a_j)\}\in\xk\subset\ell^\infty,
\]
\eqref{eq:f2-local} gives uniformly bounded local $L^q$ masses of $f_2$.
The pointwise estimate preceding \eqref{eq:df1-local} also gives
\[
\|\bar\partial f_1\|_{L^\infty}
\lesssim
\|\{G_{q,8\rho}f(a_j)\}\|_{\ell^\infty}<\infty.
\]
For each fixed $z$, kernel decay and a lattice partition give
\begin{align*}
\int_{\C^n}|f_2(w)k_z(w)|e^{-\varphi(w)}\dd v(w)
&\lesssim\sum_j e^{-c|z-a_j|}M_{q,\rho}f_2(a_j)<\infty,\\
\int_{\C^n}|f_2(w)k_z(w)|^qe^{-q\varphi(w)}\dd v(w)
&\lesssim
\sum_j e^{-cq|z-a_j|}M_{q,\rho}f_2(a_j)^q<\infty,\\
\int_{\C^n}|\bar\partial f_1(w)|^q|k_z(w)|^q
e^{-q\varphi(w)}\dd v(w)
&\lesssim
\|\bar\partial f_1\|_{L^\infty}^q
\sum_j e^{-cq|z-a_j|}<\infty.
\end{align*}
Thus $f_2\in\mathcal D_\varphi$, and $f_1=f-f_2\in\mathcal D_\varphi$.
\end{proof}

We record the endpoint estimate used below.  Let $A_\varphi$ be the
Berndtsson--Andersson solution operator from
\cite[Lemma~2.4]{HuVirtanen2023}.  Thus
$\bar\partial A_\varphi\omega=\omega$ for every admissible
$\bar\partial$-closed $(0,1)$-form $\omega$.  For $x\in\C^n$, define
\begin{equation*}
\mathcal K(x)
=
\left(|x|^{-1}+|x|^{1-2n}\right)e^{-c_\varphi|x|^2},
\end{equation*}
where $c_\varphi>0$ depends only on the lower Hessian bound in
\eqref{eq:hessian}.
Here a $(0,1)$-form has the form
$\omega=\sum_{j=1}^n\omega_j\,d\bar z_j$, and
$|\omega|=(\sum_j|\omega_j|^2)^{1/2}$.  It is
$\bar\partial$-closed when $\bar\partial\omega=0$ in the distributional
sense.  The word \emph{admissible} means that $\omega$ belongs to the
domain of the cited solution operator.  Below we only use
$\omega=g\,\bar\partial u$, for which the displayed kernel estimate is
valid.
In the application, $u=f_1$ is smooth,
$\bar\partial f_1\in L^\infty$, and $g\in\Gamma$ is holomorphic.  Hence
\[
\bar\partial(g\,\bar\partial f_1)=0
\]
in the sense of distributions.  The kernel estimates proved in
\cref{lem:ida-decomposition} also place
$g\,\bar\partial f_1$ in the domain of $A_\varphi$.
The kernel estimate for $A_\varphi$ is
\begin{equation*}
|A_\varphi\omega(z)|e^{-\varphi(z)}
\lesssim
\int_{\C^n}|\omega(\xi)|e^{-\varphi(\xi)}
\mathcal K(z-\xi)\dd v(\xi),
\end{equation*}
for every such form $\omega$.  In polar coordinates,
\begin{align*}
\|\mathcal K\|_{L^1}
&\lesssim
\int_0^1\left(t^{2n-2}+1\right)\dd t
+\int_1^\infty
\left(t^{2n-2}+1\right)e^{-c_\varphi t^2}\dd t\\
&<\infty.
\end{align*}
At $q=1$, Tonelli's theorem gives the endpoint estimate directly:
\begin{align}
\|A_\varphi\omega\|_{1,\varphi}
&\le C\int_{\C^n}\int_{\C^n}
|\omega(\xi)|e^{-\varphi(\xi)}
\mathcal K(z-\xi)\dd v(\xi)\dd v(z)\notag\\
&=C\|\mathcal K\|_{L^1}
\|\omega\|_{1,\varphi}.
\label{eq:Aphi-L1-endpoint}
\end{align}
For $1<q\le\infty$, Young's inequality gives the same conclusion.  Thus
\begin{equation}\label{eq:Aphi-Lq}
\|A_\varphi\omega\|_{q,\varphi}
\lesssim\|\omega\|_{q,\varphi},
\qquad 1\le q\le\infty.
\end{equation}
For $u\in C^1(\C^n)\cap\mathcal D_\varphi$ with
$\bar\partial u\in L^\infty$, \cite[Corollary~2.5]{HuVirtanen2023} gives
\[
H_ug
=A_\varphi(g\bar\partial u)-P A_\varphi(g\bar\partial u),
\qquad g\in\Gamma.
\]
Since $P$ is bounded on $L_\varphi^q$, \eqref{eq:Aphi-Lq} gives
\begin{equation}\label{eq:canonical-solution}
\|H_ug\|_{q,\varphi}
\lesssim
\|g|\bar\partial u|\|_{q,\varphi},
\qquad 1\le q<\infty,\quad g\in\Gamma.
\end{equation}
In particular,
\begin{equation}\label{eq:canonical-solution-q-one}
\|H_ug\|_{1,\varphi}
\le C\int_{\C^n}|g(w)|\,|\bar\partial u(w)|
e^{-\varphi(w)}\dd v(w).
\end{equation}
The constants in \eqref{eq:Aphi-L1-endpoint} and
\eqref{eq:canonical-solution-q-one} are independent of ball centers,
lattices, and lattice translates.
For a general $u$,
\begin{equation}\label{eq:rough-hankel}
\|H_ug\|_{q,\varphi}
\le (1+\|P\|_{L_\varphi^q\to L_\varphi^q})\|ug\|_{q,\varphi}.
\end{equation}

The IDA decomposition, the endpoint solution estimate, and
\cref{thm:carleson} now give the Hankel characterization.

\begin{theorem}[Absolutely summing Hankel operators]\label{thm:hankel}
Let $1\le p,q,r<\infty$ and fix $\delta>0$.  If
$f\in\mathcal D_\varphi\cap L^q_{\loc}$, then
\[
H_f\in\Piabs_r(F_\varphi^p,L_\varphi^q)
\quad\Longleftrightarrow\quad
G_{q,\delta}f\in\Xk.
\]
Moreover,
\begin{equation*}
\pi_r(H_f)\asymp\|G_{q,\delta}f\|_{\Xk}.
\end{equation*}
\end{theorem}

\begin{proof}
\emph{Sufficiency.}
Assume first that $G_{q,\delta}f\in\Xk$.  Apply \cref{lem:ida-decomposition} and put
\[
\dd\mu_1=|\bar\partial f_1|^q\dd v,
\qquad
\dd\mu_2=|f_2|^q\dd v.
\]
Equations \eqref{eq:ida-decomposition-norm} and \eqref{eq:mu-density} show that
\[
\|\widehat{\mu_1}_{q,\delta}\|_{\Xk}
+\|\widehat{\mu_2}_{q,\delta}\|_{\Xk}
\lesssim\|G_{q,R}f\|_{\Xk}.
\]
By \cref{thm:carleson}, the embeddings $J_{\mu_1}$ and $J_{\mu_2}$ are
$r$-summing.  Let $g_1,\ldots,g_N\in\Gamma$.  Equation
\eqref{eq:canonical-solution} gives
\[
\left(\sum_{j=1}^N\|H_{f_1}g_j\|_{q,\varphi}^r\right)^{1/r}
\lesssim
\left(\sum_{j=1}^N\|J_{\mu_1}g_j\|_{L_\varphi^q(\mu_1)}^r\right)^{1/r}.
\]
The right side is bounded by $\pi_r(J_{\mu_1})$ times the weak $r$-norm of $(g_j)$.  Thus
\[
\pi_r(H_{f_1})\lesssim\pi_r(J_{\mu_1}).
\]
For the second part, \eqref{eq:rough-hankel} gives
\[
\|H_{f_2}g\|_{q,\varphi}
\lesssim\|g f_2\|_{q,\varphi}
=\|J_{\mu_2}g\|_{L_\varphi^q(\mu_2)}.
\]
Apply the $r$-summing inequality for $J_{\mu_2}$ to a finite family.
This gives
\[
\pi_r(H_{f_2})\lesssim\pi_r(J_{\mu_2}).
\]
Consequently,
\begin{equation}\label{eq:hankel-upper}
\pi_r(H_f)
\le\pi_r(H_{f_1})+\pi_r(H_{f_2})
\lesssim\|G_{q,\delta}f\|_{\Xk}.
\end{equation}
\emph{The endpoint $q=1$.}
For $q=1$, the preceding estimate uses no reflexive duality.  Indeed,
\eqref{eq:canonical-solution-q-one} and
\eqref{eq:rough-hankel} give, for every finite family
$g_1,\ldots,g_N\in\Gamma$,
\begin{align*}
\left(\sum_{j=1}^N\|H_{f_1}g_j\|_{1,\varphi}^r\right)^{1/r}
&\lesssim
\left(\sum_{j=1}^N
\|g_j|\bar\partial f_1|\|_{1,\varphi}^r\right)^{1/r},\\
\left(\sum_{j=1}^N\|H_{f_2}g_j\|_{1,\varphi}^r\right)^{1/r}
&\lesssim
\left(\sum_{j=1}^N
\|g_jf_2\|_{1,\varphi}^r\right)^{1/r}.
\end{align*}
The two right-hand sides are bounded by
\[
C\bigl(\pi_r(J_{\mu_1})+\pi_r(J_{\mu_2})\bigr)
w_r((g_j);F_\varphi^p).
\]
Thus the endpoint follows from the same Carleson theorem, with constants
independent of the centers and of lattice translates.

\emph{Necessity.}
Conversely, suppose that $H_f$ is $r$-summing.  Choose $\rho>0$ so that
\eqref{eq:kernel-lower} holds on $B(a,4\rho)$.  Let $\{a_j\}$ be a
$\rho$-lattice and write it as
$\Lambda_1\cup\cdots\cup\Lambda_N$, with each subfamily sufficiently
separated.  The function $k_{a_j}$ has no zeros on $B(a_j,4\rho)$.  Hence
\[
Q_j(z)=\frac{P(fk_{a_j})(z)}{k_{a_j}(z)}
\]
is holomorphic there.  Using \eqref{eq:kernel-lower},
\begin{align*}
\|H_fk_{a_j}\|_{L^q(B(a_j,4\rho),e^{-q\varphi}\dd v)}^q
&=\int_{B(a_j,4\rho)}|f-Q_j|^q|k_{a_j}|^qe^{-q\varphi}\dd v\\
&\gtrsim\int_{B(a_j,4\rho)}|f-Q_j|^q\dd v\\
&\gtrsim G_{q,4\rho}f(a_j)^q.
\end{align*}
On each separated subfamily, apply \cref{prop:diagonal-extraction} with
$A=H_f$ and $\mu=v$.  The local $L^q$ norm above is at most the extracted
coefficient.  The ideal property and \cref{thm:diagonal} give
\[
\|\{G_{q,4\rho}f(a_j):a_j\in\Lambda_\nu\}\|_{\xk}
\lesssim\pi_r(H_f).
\]
The estimate is uniform in $1\le\nu\le N$.  Combine the subfamilies and
use the triangle inequality in $\xk$.  Discretization and radius
independence then give
\begin{equation}\label{eq:hankel-lower}
\|G_{q,\delta}f\|_{\Xk}\lesssim\pi_r(H_f).
\end{equation}
Combining \eqref{eq:hankel-upper} and \eqref{eq:hankel-lower} proves the theorem.
\end{proof}

\subsection{Operator diagonals and scalar recovery}

The IDA seminorm does not see an entire summand.  We recover that summand
from the Berezin transform or a complex ball average.  The local quotient
construction and all sequence-space estimates are proved below.  This
includes the logarithmic Orlicz endpoint.

For $f\in\mathcal D_\varphi\cap L^q_{\loc}$, put, whenever the norm is finite,
\begin{equation*}
E_f(z)=\|H_fk_z\|_{q,\varphi}\in[0,\infty].
\end{equation*}
Thus $E_f(z)=\infty$ when the core vector $H_fk_z$ does not belong to
$L_\varphi^q$.

We first treat an arbitrary operator.  The next estimate is the continuous
Fock-space version of taking the diagonal of a matrix from $\ell^p$ to
$\ell^q$.

\begin{proposition}\label{prop:operator-diagonal}
Let $1\le p<\infty$, $1\le q<\infty$, $1\le r<\infty$, and let $A\in\Piabs_r(F_\varphi^p,F_\varphi^q)$.  Define
\[
d_A(z)=\frac{(Ak_z)(z)}{K(z,z)^{1/2}}.
\]
Then $d_A\in\Xk$ and
\begin{equation}\label{eq:operator-diagonal}
\|d_A\|_{\Xk}\lesssim\pi_r(A).
\end{equation}
For every sufficiently separated sequence $\Lambda=\{a_j\}$,
\begin{equation}\label{eq:operator-diagonal-discrete}
\|\{d_A(a_j)\}\|_{\xk}\lesssim\pi_r(A).
\end{equation}
The discrete constant depends only on the fixed separation parameter and
the structural constants in \eqref{eq:hessian}; it is independent of $A$
and of translations of $\Lambda$.
In particular, if $A=T_f$, then $d_A=\widetilde f$.
\end{proposition}

\begin{proof}
\emph{Separated samples.}
Let $\Lambda=\{a_j\}$ be sufficiently separated.  Kernel decay and the weighted submean inequality give bounded maps
\begin{align*}
S_\Lambda:\ell^p&\longrightarrow F_\varphi^p,
&S_\Lambda c&=\sum_jc_jk_{a_j},\\
U_\Lambda:F_\varphi^q&\longrightarrow\ell^q,
&U_\Lambda g&=\left\{\frac{g(a_j)}{K(a_j,a_j)^{1/2}}\right\}_j.
\end{align*}
For the evaluation map, choose disjoint balls $B(a_j,\rho)$ and use the weighted submean estimate:
\[
\sum_j\left|\frac{g(a_j)}{K(a_j,a_j)^{1/2}}\right|^q
\lesssim
\sum_j\int_{B(a_j,\rho)}|g(w)|^qe^{-q\varphi(w)}\dd v(w)
\lesssim\|g\|_{q,\varphi}^q.
\]

Set $B_\Lambda=U_\Lambda AS_\Lambda$.  Then
\begin{equation*}
\pi_r(B_\Lambda)
\le\|U_\Lambda\|\pi_r(A)\|S_\Lambda\|
\lesssim\pi_r(A),
\end{equation*}
and the diagonal entries of $B_\Lambda$ are $d_A(a_j)$.  For a finite set
$J\subset\N$, let $P_J$ be the coordinate projection and let
$D_\varepsilon$ be multiplication by independent signs.  For a finite
matrix $B$, $\diag(B)$ denotes the diagonal operator having the same
diagonal entries as $B$.  Rademacher averaging gives
\[
\diag(P_JB_\Lambda P_J)
=\int_0^1D_{\varepsilon(t)}P_JB_\Lambda P_JD_{\varepsilon(t)}\dd t.
\]
Thus
\begin{equation*}
\pi_r\bigl(D_{\{d_A(a_j):j\in J\}}\bigr)
\le\int_0^1
\pi_r(D_{\varepsilon(t)}P_JB_\Lambda P_JD_{\varepsilon(t)})\dd t
\le\pi_r(B_\Lambda).
\end{equation*}
We now pass from finite to infinite diagonals.  The estimate for an
infinite diagonal is at least the estimate for each finite compression.
For the reverse inequality, first take vectors in $c_{00}$.  One finite
set $J$ contains all their supports, so the finite estimate applies.
For general vectors in $\ell^p$, truncate their coordinates and use
Fatou's lemma in the target $\ell^q$.  Therefore
\begin{equation*}
\pi_r(D_{\{d_A(a_j)\}})
=\sup_{J\subset\N,\,J\text{ finite}}
\pi_r(D_{\{d_A(a_j):j\in J\}})
\lesssim\pi_r(A).
\end{equation*}
By \eqref{eq:intrinsic-diagonal-ideal}, this is precisely
\eqref{eq:operator-diagonal-discrete}.

\emph{Passage to the continuous norm.}
Kernel continuity implies that $z\mapsto k_z$ is continuous in
$F_\varphi^p$.  Weighted point evaluations are locally uniform on bounded
subsets of $F_\varphi^q$.  Hence $d_A$ is continuous.  Equation
\eqref{eq:kernel-norm} and the point-evaluation estimate give
\begin{equation*}
|d_A(z)|
\lesssim\|Ak_z\|_{q,\varphi}
\le\|A\|\|k_z\|_{p,\varphi}
\lesssim\|A\|,
\qquad z\in\C^n.
\end{equation*}
Let
$\{E_j\}$ be the reference partition in
\eqref{eq:reference-partition}.  Choose $z_j\in E_j$ so that
\begin{equation*}
\operatorname*{ess\,sup}_{z\in E_j}|d_A(z)|
\le2|d_A(z_j)|.
\end{equation*}
Partition the indices into finitely many sets
$I_1,\ldots,I_N$ such that each sequence $\{z_j:j\in I_\nu\}$ is sufficiently separated.  Apply
\eqref{eq:operator-diagonal-discrete} to every subsequence.  Solidity and the triangle inequality give
\begin{align*}
\|d_A\|_{\Xk}
&\le2\left\|\{|d_A(z_j)|\}_j\right\|_{\xk}\\
&\le2\sum_{\nu=1}^N
\left\|\{|d_A(z_j)|\mathbf1_{I_\nu}(j)\}_j\right\|_{\xk}
\lesssim\pi_r(A).
\end{align*}
This proves \eqref{eq:operator-diagonal}.
If $A=T_f$, then \eqref{eq:berezin-evaluation} and the definition of
$d_A$ give
\[
d_{T_f}(z)=\widetilde f(z).
\]
\end{proof}

The complex ball average admits a direct recovery argument.  Unlike the Berezin transform, it does not require kernel quotients.

\begin{lemma}\label{lem:complex-average}
Let $1\le p,q,r<\infty$ and $f\in L^q_{\loc}$.  For every fixed $\delta>0$,
define $A_\rho u=\avg_{B(\cdot,\rho)}u\dd v$.  Then
\begin{equation}\label{eq:complex-average-pointwise}
M_{q,\delta}f(z)
\lesssim G_{q,4\delta}f(z)
+A_{2\delta}(|a_\delta f|)(z).
\end{equation}
Consequently,
\begin{equation*}
\|M_{q,\delta}f\|_{\Xk}
\asymp
\|a_\delta f\|_{\Xk}+\|G_{q,\delta}f\|_{\Xk}.
\end{equation*}
\end{lemma}

\begin{proof}
Fix $z$ and choose $h\in\Hol(B(z,4\delta))$ such that
\[
\left(\avg_{B(z,4\delta)}|f-h|^q\dd v\right)^{1/q}
\le2G_{q,4\delta}f(z).
\]
If $w\in B(z,2\delta)$, then $B(w,\delta)\subset B(z,4\delta)$.  The mean-value property and H\"older's inequality give
\[
|h(w)-a_\delta f(w)|
=|a_\delta(h-f)(w)|
\lesssim G_{q,4\delta}f(z).
\]
The holomorphic submean inequality, now applied to $h$ on $B(z,2\delta)$, yields
\begin{align*}
M_{q,\delta}f(z)
&\le M_{q,\delta}(f-h)(z)+M_{q,\delta}h(z)\\
&\lesssim G_{q,4\delta}f(z)
+\avg_{B(z,2\delta)}|h(w)|\dd v(w)\\
&\lesssim G_{q,4\delta}f(z)+A_{2\delta}(|a_\delta f|)(z).
\end{align*}
Taking the $\Xk$-norm in \eqref{eq:complex-average-pointwise} gives
\[
\|M_{q,\delta}f\|_{\Xk}
\lesssim
\|G_{q,4\delta}f\|_{\Xk}
+\|A_{2\delta}(|a_\delta f|)\|_{\Xk}.
\]
By \cref{lem:discretization} and \eqref{eq:convolution-Xk},
\[
\|M_{q,\delta}f\|_{\Xk}
\lesssim
\|G_{q,\delta}f\|_{\Xk}+\|a_\delta f\|_{\Xk}.
\]
The reverse bound follows from
\[
|a_\delta f|\le M_{q,\delta}f,
\qquad
G_{q,\delta}f\le M_{q,\delta}f.
\]
\end{proof}

To control the error in scalar recovery, we first estimate the Hankel
columns $H_fk_z$.

\begin{lemma}\label{lem:hankel-columns}
Let $1\le p,q,r<\infty$.  Fix $R>0$.
Suppose that
\[
f\in\mathcal D_\varphi\cap L^q_{\loc},
\qquad G_{q,R}f\in\Xk.
\]
There are $R_1>R$ and
$c,C>0$ such that
\begin{equation}\label{eq:hankel-column-convolution}
E_f(z)^q
\le C\int_{\C^n}e^{-c|z-w|}G_{q,R_1}f(w)^q\dd v(w).
\end{equation}
Consequently,
\begin{equation}\label{eq:E-norm}
\|E_f\|_{\Xk}\lesssim\|G_{q,R}f\|_{\Xk}.
\end{equation}
\end{lemma}

\begin{proof}
Use the decomposition $f=f_1+f_2$ from
\cref{lem:ida-decomposition}.  Combine \eqref{eq:kernel-upper}, the
canonical solution estimate, and \eqref{eq:rough-hankel}.  This gives
\[
E_f(z)^q
\lesssim
\int_{\C^n}e^{-c|z-w|}
\bigl(|\bar\partial f_1(w)|^q+|f_2(w)|^q\bigr)\dd v(w).
\]
Let $\{a_m\}$ be the lattice in \cref{lem:ida-decomposition}.  After a
fixed enlargement of the radius, put
\[
g_m=G_{q,R_0}f(a_m),
\qquad
s_m=\sum_{a_\ell\in B(a_m,R_0)}g_\ell.
\]
By \eqref{eq:ida-decomposition} and finite overlap,
\begin{align*}
\int_{B(a_m,\rho)}
\bigl(|\bar\partial f_1|^q+|f_2|^q\bigr)\dd v
&\lesssim s_m^q,\\
s_m^q
&\lesssim
\sum_{a_\ell\in B(a_m,R_0)}g_\ell^q.
\end{align*}
Therefore
\begin{align*}
E_f(z)^q
&\lesssim
\sum_m e^{-c|z-a_m|}
\sum_{a_\ell\in B(a_m,R_0)}g_\ell^q\\
&\lesssim
\sum_\ell e^{-c_1|z-a_\ell|}g_\ell^q.
\end{align*}
For $w\in B(a_\ell,\rho)$, ball inclusion gives
\[
g_\ell\lesssim G_{q,R_1}f(w)
\]
with a fixed $R_1>R_0+\rho$.  Hence
\[
e^{-c_1|z-a_\ell|}g_\ell^q
\lesssim
\int_{B(a_\ell,\rho)}
e^{-c_2|z-w|}G_{q,R_1}f(w)^q\dd v(w).
\]
Summing in $\ell$ proves \eqref{eq:hankel-column-convolution}.

To prove \eqref{eq:E-norm}, discretize on a unit lattice $\{a_j\}$.
By ball inclusion, there is a fixed $R_2>R_1$ such that
\[
G_{q,R_1}f(w)\lesssim G_{q,R_2}f(a_m),
\qquad w\in B(a_m,1).
\]
If
\[
e_j=\sup_{B(a_j,1)}E_f,
\qquad
g_j=G_{q,R_2}f(a_j),
\]
then
\[
e_j
\le C\left(\sum_m e^{-c|a_j-a_m|}g_m^q\right)^{1/q}
\le C\sum_m e^{-c|a_j-a_m|/q}g_m.
\]
Convolution by the exponentially decreasing sequence is bounded on
$\xk$.  Indeed, if $\tau_\ell$ is a coordinate translation, symmetry and the triangle inequality give
\[
\left\|\sum_\ell e^{-c|\ell|/q}\tau_\ell g\right\|_{\xk}
\le
\sum_\ell e^{-c|\ell|/q}\|\tau_\ell g\|_{\xk}
\lesssim\|g\|_{\xk}.
\]
Thus
\[
\|e\|_{\xk}\lesssim\|g\|_{\xk}.
\]
If $\{E_j\}$ is the lattice partition from \cref{lem:discretization}, then
$E_f\le\sum_je_j\mathbf1_{E_j}$.  Definition~\eqref{eq:intrinsic-function-space} and solidity give
\[
\|E_f\|_{\Xk}\lesssim\|e\|_{\xk}.
\]
Finally, radius independence gives
\[
\|g\|_{\xk}
\lesssim\|G_{q,R_2}f\|_{\Xk}
\asymp\|G_{q,R}f\|_{\Xk}.
\]
This proves \eqref{eq:E-norm}.
\end{proof}

The next lemma recovers the local size of $f$ from the scalar diagonal
$\widetilde f$ and the controlled error $E_f$.

\begin{lemma}[Local scalar recovery]\label{lem:local-recovery}
Let $1\le q<\infty$ and let
$f\in\mathcal D_\varphi\cap L^q_{\loc}$.  Suppose that $E_f(z)<\infty$
for every $z$.  There is a sufficiently small $\rho>0$ such that
\begin{equation}\label{eq:local-recovery}
M_{q,\rho}f(z)
\le C\left(
E_f(z)
+A_{4\rho}(|\widetilde f|)(z)
+A_{4\rho}E_f(z)
\right).
\end{equation}
\end{lemma}

\begin{proof}
Choose $\rho$ so that $k_z$ has no zeros in $B(z,16\rho)$, uniformly in $z$.  Define
\[
Q_z(w)=\frac{T_fk_z(w)}{k_z(w)},
\qquad w\in B(z,16\rho).
\]
Then $Q_z$ is holomorphic, $Q_z(z)=\widetilde f(z)$, and
\begin{equation}\label{eq:f-Q}
\left(\avg_{B(z,16\rho)}|f-Q_z|^q\dd v\right)^{1/q}
\lesssim E_f(z).
\end{equation}
If $w\in B(z,4\rho)$, then $B(w,\rho)$ lies in the domains of both $Q_z$ and $Q_w$.  On this ball,
\[
Q_z-Q_w=(Q_z-f)+(f-Q_w).
\]
The holomorphic submean inequality and \eqref{eq:f-Q} therefore give
\begin{equation}\label{eq:Q-difference}
\begin{aligned}
|Q_z(w)-\widetilde f(w)|
&=|Q_z(w)-Q_w(w)|\\
&\lesssim
\left(\avg_{B(w,\rho)}|Q_z-Q_w|^q\dd v\right)^{1/q}\\
&\le
\left(\avg_{B(w,\rho)}|Q_z-f|^q\dd v\right)^{1/q}
+\left(\avg_{B(w,\rho)}|f-Q_w|^q\dd v\right)^{1/q}\\
&\lesssim E_f(z)+E_f(w).
\end{aligned}
\end{equation}
Since $Q_z$ is holomorphic, its local $L^q$ norm on $B(z,\rho)$ is controlled by its local $L^1$ norm on $B(z,4\rho)$:
\[
\left(\avg_{B(z,\rho)}|Q_z|^q\dd v\right)^{1/q}
\le \sup_{B(z,\rho)}|Q_z|
\lesssim \avg_{B(z,4\rho)}|Q_z(w)|\dd v(w).
\]
Equation \eqref{eq:Q-difference} therefore gives
\[
\left(\avg_{B(z,\rho)}|Q_z|^q\dd v\right)^{1/q}
\lesssim
A_{4\rho}(|\widetilde f|)(z)+E_f(z)+A_{4\rho}E_f(z).
\]
Adding \eqref{eq:f-Q} proves \eqref{eq:local-recovery}.
\end{proof}

Combining the complex-average and Berezin-transform recoveries gives the
required scalar equivalence.

\begin{proposition}\label{prop:scalar-equivalence}
Let $1\le p,q,r<\infty$ and fix $\delta>0$.  Let
$f\in\mathcal D_\varphi\cap L^q_{\loc}$ and assume that
$G_{q,\delta}f\in\Xk$.  Then
\begin{equation*}
\begin{split}
\|M_{q,\delta}f\|_{\Xk}
&\asymp \|\widetilde f\|_{\Xk}+\|G_{q,\delta}f\|_{\Xk}\\
&\asymp \|a_\delta f\|_{\Xk}+\|G_{q,\delta}f\|_{\Xk}.
\end{split}
\end{equation*}
\end{proposition}

\begin{proof}
The fixed-ball averages in \eqref{eq:local-recovery} are bounded on
$\Xk$ by \eqref{eq:convolution-Xk}.  Apply
\cref{lem:hankel-columns,lem:local-recovery}.  We obtain
\[
\|M_{q,\delta}f\|_{\Xk}
\lesssim
\|\widetilde f\|_{\Xk}+\|G_{q,R}f\|_{\Xk}.
\]
Radius independence gives the first upper bound.  The complex-average
bound follows from \cref{lem:complex-average}.  It keeps the prescribed
averaging radius.

Conversely, \eqref{eq:kernel-upper} gives
\[
|\widetilde f(z)|
\lesssim
\int_{\C^n}e^{-c|z-w|}|f(w)|\dd v(w).
\]
Let $\{a_j\}$ be a unit lattice and let $z$ belong to its $j$th partition cell.  H\"older's inequality on each cell gives
\[
|\widetilde f(z)|
\lesssim\sum_m e^{-c|a_j-a_m|}M_{q,R}f(a_m).
\]
The discrete exponential kernel belongs to $\ell^1$.  Symmetry and the triangle inequality in $\xk$, followed by
\cref{lem:discretization}, therefore give
\[
\|\widetilde f\|_{\Xk}
\lesssim
\|M_{q,R}f\|_{\Xk}.
\]
Finally, $|a_\delta f|\le M_{q,\delta}f$, \eqref{eq:G-less-M}, and radius independence complete the proof.
\end{proof}

The preceding intrinsic estimate has the following explicit function-space
form for every parameter triple.

\begin{corollary}
\label{cor:explicit-scalar-recovery}
Let $1\le p,q,r<\infty$ and fix $\delta>0$.
Put $\kappa=\kappa(p,q,r)$.  If
\[
f\in\mathcal D_\varphi\cap L^q_{\loc},
\qquad
G_{q,\delta}f\in\mathfrak E_\kappa,
\]
then
\begin{equation*}
\begin{split}
\|M_{q,\delta}f\|_{\mathfrak E_\kappa}
&\asymp
\|\widetilde f\|_{\mathfrak E_\kappa}
+\|G_{q,\delta}f\|_{\mathfrak E_\kappa}\\
&\asymp
\|a_\delta f\|_{\mathfrak E_\kappa}
+\|G_{q,\delta}f\|_{\mathfrak E_\kappa}.
\end{split}
\end{equation*}
\end{corollary}

\begin{proof}
Every convolution kernel used in
\cref{lem:complex-average,lem:hankel-columns,lem:local-recovery} is bounded
on $\mathfrak E_\kappa$.  Young's inequality gives the $L^s$ case.  The
$L^\infty$ case is immediate.  For $L^{q^-}$, normalize the relevant
convolution kernel $\mathcal L$ by $\|\mathcal L\|_1=1$.  Jensen's
inequality gives
\[
\Psi_q\!\left(\frac{\mathcal L*u(z)}{\lambda}\right)
\le
\mathcal L*\Psi_q\!\left(\frac{u}{\lambda}\right)(z),
\]
and Tonelli's theorem gives the Luxemburg estimate.

The proof of \cref{lem:hankel-columns}, with
$\mathfrak E_\kappa$ in place of $\Xk$, yields
\begin{equation}\label{eq:E-explicit}
\|E_f\|_{\mathfrak E_\kappa}
\lesssim
\|G_{q,\delta}f\|_{\mathfrak E_\kappa}.
\end{equation}
Apply the corresponding function-space norm to
\eqref{eq:local-recovery}.  Using \eqref{eq:E-explicit},
\begin{equation*}
\|M_{q,\delta}f\|_{\mathfrak E_\kappa}
\lesssim
\|\widetilde f\|_{\mathfrak E_\kappa}
+\|G_{q,\delta}f\|_{\mathfrak E_\kappa}.
\end{equation*}
The kernel estimate \eqref{eq:kernel-upper} and Young's inequality give
\begin{equation*}
\|\widetilde f\|_{\mathfrak E_\kappa}
\lesssim
\|M_{q,\delta}f\|_{\mathfrak E_\kappa}.
\end{equation*}
Together with $G_{q,\delta}f\le M_{q,\delta}f$, this proves the first comparison.  The second follows from
\eqref{eq:complex-average-pointwise},
$|a_\delta f|\le M_{q,\delta}f$, and the same convolution estimates.
\end{proof}

\section{Absolutely summing Toeplitz operators}\label{sec:toeplitz}

We now prove the joint theorem and then extract the IDA-symbol formulation.

\begin{proof}[Proof of \cref{thm:intro-main}]
\emph{Graph equivalence.}
Define
\[
W:L_\varphi^q\longrightarrow F_\varphi^q\oplus_qL_\varphi^q,
\qquad
Wu=(Pu,(I-P)u).
\]
The boundedness of $P$ gives
\begin{equation*}
\|Wu\|_{F_\varphi^q\oplus_qL_\varphi^q}
\le C_q\|u\|_{q,\varphi},
\qquad
\|u\|_{q,\varphi}
\le \|Pu\|_{q,\varphi}+\|(I-P)u\|_{q,\varphi}.
\end{equation*}
Define
\[
\Sigma:F_\varphi^q\oplus_qL_\varphi^q\longrightarrow L_\varphi^q,
\qquad
\Sigma(u,v)=\iota u+v.
\]
On $\Gamma$,
\begin{equation*}
\mathscr T_f=WM_f,
\qquad
M_f=\Sigma\mathscr T_f.
\end{equation*}
Hence the ideal property yields
\begin{equation}\label{eq:graph-norm-comparison}
\pi_r(\mathscr T_f)\le C_q\pi_r(M_f),
\qquad
\pi_r(M_f)\le\|\Sigma\|\pi_r(\mathscr T_f).
\end{equation}
Moreover, \eqref{eq:finite-direct-sum} gives
\begin{equation}\label{eq:pair-graph-comparison}
\max\{\pi_r(T_f),\pi_r(H_f)\}
\le\pi_r(\mathscr T_f)
\le\pi_r(T_f)+\pi_r(H_f)
\le2\pi_r(\mathscr T_f).
\end{equation}
Equations \eqref{eq:graph-norm-comparison} and
\eqref{eq:pair-graph-comparison} prove the formal equivalence of (1)--(3).

\emph{Identification of the extensions.}
We verify that the extensions coincide with the operators determined by the
symbol.  Suppose first that $M_f$ is $r$-summing.  Then
\[
T_f=PM_f,
\qquad
H_f=(I-P)M_f
\]
on $\Gamma$, and both right-hand sides extend boundedly.  Conversely, suppose that $T_f$ and $H_f$ extend.  On $\Gamma$,
\begin{equation}\label{eq:multiplication-on-core}
M_f=\iota T_f+H_f.
\end{equation}
To identify the extension, take $g_m\in\Gamma$ with
$g_m\to g$ in $F_\varphi^p$.  The right-hand side of
\eqref{eq:multiplication-on-core} converges to some $u\in L_\varphi^q$.
Since $\varphi$ is bounded on $B(0,R)$,
\[
\|fg_m-u\|_{L^q(B(0,R))}
\le C_R\|fg_m-u\|_{q,\varphi}
\longrightarrow0.
\]
Also,
$g_m\to g$ locally uniformly.  Since $f\in L^q_{\loc}$,
\begin{equation*}
\|fg_m-fg\|_{L^q(B(0,R))}
\le
\|f\|_{L^q(B(0,R))}
\sup_{B(0,R)}|g_m-g|
\longrightarrow0.
\end{equation*}
Uniqueness of the local $L^q$ limit gives $u=fg$ on every ball.  Thus the
extension is multiplication by $f$.

\emph{Local criteria.}
By \cref{cor:multiplication},
\[
\pi_r(M_f)\asymp\|M_{q,\delta}f\|_{\Xk}.
\]
This proves (1)$\Longleftrightarrow$(4).  For every $1\le q<\infty$,
\cref{lem:complex-average} gives
\[
\|M_{q,\delta}f\|_{\Xk}
\asymp
\|a_\delta f\|_{\Xk}+\|G_{q,\delta}f\|_{\Xk},
\]
so (4)$\Longleftrightarrow$(5).  This also completes
\eqref{eq:intro-main-norm}.

Finally, condition (4) implies
$G_{q,\delta}f\in\Xk$ by \eqref{eq:G-less-M}; condition (6) contains this
membership.  Under either condition, \cref{prop:scalar-equivalence} gives
\[
\|M_{q,\delta}f\|_{\Xk}
\asymp
\|\widetilde f\|_{\Xk}+\|G_{q,\delta}f\|_{\Xk}.
\]
Hence (4)$\Longleftrightarrow$(6), and
\eqref{eq:intro-main-scalar-norm} follows.
\end{proof}

\begin{proof}[Proof of \cref{cor:intro-ida}]
By \cref{thm:hankel},
\begin{equation}\label{eq:Hf-summing}
H_f\in\Piabs_r(F_\varphi^p,L_\varphi^q),
\qquad
\pi_r(H_f)\asymp\|G_{q,\delta}f\|_{\Xk}.
\end{equation}
On the common core $\Gamma$,
\[
M_f=\iota T_f+H_f,
\qquad
T_f=PM_f.
\]
The density of $\Gamma$ gives the same identities for the unique bounded
extensions.  Hence
\begin{align}
\pi_r(M_f)
&\le \pi_r(T_f)+\pi_r(H_f)
\lesssim \pi_r(T_f)+\|G_{q,\delta}f\|_{\Xk},
\notag\\
\pi_r(T_f)
&\le \|P\|_{L_\varphi^q\to F_\varphi^q}\pi_r(M_f).
\label{eq:matching-reverse}
\end{align}
Thus
\[
T_f\in\Piabs_r
\quad\Longleftrightarrow\quad
M_f\in\Piabs_r
\quad\Longleftrightarrow\quad
M_{q,\delta}f\in\Xk.
\]
Proposition~\ref{prop:scalar-equivalence} then gives the equivalent
$a_\delta f$ and $\widetilde f$ conditions.  Combining
\eqref{eq:Hf-summing}--\eqref{eq:matching-reverse} with
\eqref{eq:multiplication-norm} proves \eqref{eq:intro-ida-norm}.
\end{proof}

\section{Consequences, endpoints, and sharpness}\label{sec:consequences}

We first record consequences that can be used without referring to the proof of the main theorem.

\begin{corollary}\label{cor:one-sided}
Let $1\le p,q,r<\infty$, fix $\delta>0$, and let
$f\in\mathcal D_\varphi\cap L^q_{\loc}$.  If
\[
M_{q,\delta}f\in\Xk,
\]
then $T_f\in\Piabs_r(F_\varphi^p,F_\varphi^q)$ and
\[
\pi_r(T_f)\lesssim\|M_{q,\delta}f\|_{\Xk}.
\]
No IDA assumption is required in this direction.
\end{corollary}

\begin{proof}
By \cref{cor:multiplication}, $M_f$ is $r$-summing.  Since $T_f=PM_f$, the conclusion follows from the ideal property.
\end{proof}

The converse implication needs IDA in general.  Without IDA, the operator
diagonal still gives the following necessary condition.

\begin{corollary}\label{cor:scalar-necessity}
Let $1\le p<\infty$, $1\le q<\infty$, and $1\le r<\infty$.  If
$f\in\mathcal D_\varphi$ and
\[
T_f\in\Piabs_r(F_\varphi^p,F_\varphi^q),
\]
then $\widetilde f\in\Xk$ and
\[
\|\widetilde f\|_{\Xk}\lesssim\pi_r(T_f).
\]
No IDA assumption is used.
\end{corollary}

\begin{proof}
Apply \cref{prop:operator-diagonal} to $A=T_f$.
\end{proof}

We now specialize the graph, Hankel, and Toeplitz criteria to $q=1$.

\begin{corollary}[The complete $q=1$ endpoint]
Let $1\le p,r<\infty$, let $s_p$ be given by
\eqref{eq:q-one-exponent-intro}, and let
$f\in\mathcal D_\varphi\cap L^1_{\loc}$.  Then the following assertions
hold for every $\delta>0$.
\begin{enumerate}
\item Without an IDA assumption,
\begin{equation}\label{eq:q-one-unconditional}
\begin{aligned}
M_f\in\Piabs_r(F_\varphi^p,L_\varphi^1)
&\iff \mathscr T_f\in
\Piabs_r(F_\varphi^p,F_\varphi^1\oplus_1L_\varphi^1)\\
&\iff
\left\{
\begin{array}{l}
T_f\in\Piabs_r(F_\varphi^p,F_\varphi^1),\\[-1pt]
H_f\in\Piabs_r(F_\varphi^p,L_\varphi^1),
\end{array}
\right.\\
&\iff M_{1,\delta}f\in L^{s_p}\\
&\iff a_\delta f,\ G_{1,\delta}f\in L^{s_p}\\
&\iff \widetilde f,\ G_{1,\delta}f\in L^{s_p}.
\end{aligned}
\end{equation}
The corresponding norms satisfy
\begin{equation}\label{eq:q-one-unconditional-norms}
\begin{aligned}
\pi_r(M_f)
&\asymp\pi_r(\mathscr T_f)
\asymp\pi_r(T_f)+\pi_r(H_f)\\
&\asymp\|M_{1,\delta}f\|_{L^{s_p}}\\
&\asymp\|a_\delta f\|_{L^{s_p}}
+\|G_{1,\delta}f\|_{L^{s_p}}\\
&\asymp\|\widetilde f\|_{L^{s_p}}
+\|G_{1,\delta}f\|_{L^{s_p}}.
\end{aligned}
\end{equation}

\item The Hankel part is independent of the analytic component:
\begin{equation}\label{eq:q-one-hankel}
H_f\in\Piabs_r(F_\varphi^p,L_\varphi^1)
\iff G_{1,\delta}f\in L^{s_p},
\qquad
\pi_r(H_f)\asymp\|G_{1,\delta}f\|_{L^{s_p}}.
\end{equation}

\item If $G_{1,\delta}f\in L^{s_p}$, then
\begin{equation}\label{eq:q-one-equivalences}
\begin{aligned}
T_f\in\Piabs_r(F_\varphi^p,F_\varphi^1)
&\iff M_f\in\Piabs_r(F_\varphi^p,L_\varphi^1)\\
&\iff M_{1,\delta}f\in L^{s_p}\\
&\iff a_\delta f\in L^{s_p}
\iff\widetilde f\in L^{s_p}.
\end{aligned}
\end{equation}
Moreover,
\begin{equation}\label{eq:q-one-norms}
\begin{aligned}
\pi_r(T_f)+\|G_{1,\delta}f\|_{L^{s_p}}
&\asymp\pi_r(M_f)+\|G_{1,\delta}f\|_{L^{s_p}}\\
&\asymp\|M_{1,\delta}f\|_{L^{s_p}}\\
&\asymp\|a_\delta f\|_{L^{s_p}}
+\|G_{1,\delta}f\|_{L^{s_p}}\\
&\asymp\|\widetilde f\|_{L^{s_p}}
+\|G_{1,\delta}f\|_{L^{s_p}}.
\end{aligned}
\end{equation}
\end{enumerate}
The exponent $s_p$ is independent of $r$.  The equivalence constants may
depend on $p,r,n$, the Hessian bounds, and the fixed radius $\delta$.
They are independent of $f$, of ball centers, and of lattice translates.
\end{corollary}

\begin{proof}
By \cref{prop:q-one-diagonal,lem:discretization}, for each local quantity
$u=u_{\mu,t},u_{f,t},v_{f,t}$,
\begin{equation}\label{eq:q-one-continuous-realization}
\|u\|_{\mathfrak D_{p,1}^{\,r}}
\asymp\|u\|_{L^{s_p}}.
\end{equation}
Indeed, discretization reduces both sides to the same step sequence, and
\eqref{eq:q-one-diagonal-norm} identifies its norm with
$\ell^{s_p}$.  Now \cref{thm:intro-main} and
\eqref{eq:q-one-continuous-realization} give
\eqref{eq:q-one-unconditional}--\eqref{eq:q-one-unconditional-norms}.
Theorem~\ref{thm:hankel} gives \eqref{eq:q-one-hankel}.  Under the stated
IDA condition, \cref{cor:intro-ida} gives
\eqref{eq:q-one-equivalences}--\eqref{eq:q-one-norms}.
\end{proof}

The intrinsic condition can now be written without lattice notation for the
full parameter range.

\begin{corollary}\label{cor:complete-parameter}
Let $1\le p,q,r<\infty$, fix $\delta>0$, and let
$f\in\mathcal D_\varphi\cap L^q_{\loc}$ satisfy
$G_{q,\delta}f\in\Xk$.  Put
\[
\Phi_f=M_{q,\delta}f.
\]
When $1\le p\le2$ and $1\le q\le2$, also put
$s_0=(1/p'+1/q-1/2)^{-1}$.
Define
\begin{equation}\label{eq:complete-parameter-space}
\mathcal Y_{p,q,r}=
\begin{cases}
L^{s_0},
&1\le p\le2,\ 1\le q\le2,\\[2pt]
L^\infty,
&p=1,\ 2<q<\infty,\\[2pt]
L^q,
&2<p<\infty,\ 1\le q<p',\\[2pt]
L^{q^-},
&2<p<\infty,\ q=p',\ 1\le r<q,\\[2pt]
L^q,
&2<p<\infty,\ q=p',\ q\le r<\infty,\\[2pt]
L^{\max\{p',\min\{r,q\}\}},
&2\le p<\infty,\ p'<q<\infty,\\[2pt]
L^{p'},
&1<p<2<q<\infty,\ 1\le r\le p',\\[2pt]
L^r,
&1<p<2,\ p'<q<\infty,\ p'<r\le q,
\\[2pt]
L^{\frac{p'q(r-2)}{(p'-2)(q-2)+2(r-2)}},
&1<p<2<q<\infty,\ r>\max\{p',q\},
\end{cases}
\end{equation}
Then
\begin{equation*}
T_f\in\Piabs_r(F_\varphi^p,F_\varphi^q)
\quad\Longleftrightarrow\quad
\Phi_f\in\mathcal Y_{p,q,r}.
\end{equation*}
In every line of \eqref{eq:complete-parameter-space},
\begin{equation}\label{eq:complete-parameter-norm}
\pi_r(T_f)+\|G_{q,\delta}f\|_{\mathcal Y_{p,q,r}}
\asymp
\|\Phi_f\|_{\mathcal Y_{p,q,r}}.
\end{equation}
At the fourth line, every norm in \eqref{eq:complete-parameter-norm} is interpreted as the Luxemburg norm in $L^{q^-}$.
\end{corollary}

\begin{proof}
By \cref{prop:explicit-realization},
\begin{align*}
\|M_{q,\delta}f\|_{\Xk}
&\asymp\|M_{q,\delta}f\|_{\mathcal Y_{p,q,r}},\\
\|G_{q,\delta}f\|_{\Xk}
&\asymp\|G_{q,\delta}f\|_{\mathcal Y_{p,q,r}}.
\end{align*}
Therefore \cref{cor:intro-ida} gives
\begin{align*}
T_f\in\Piabs_r
&\iff M_{q,\delta}f\in\Xk
\iff M_{q,\delta}f\in\mathcal Y_{p,q,r},\\
\pi_r(T_f)+\|G_{q,\delta}f\|_{\mathcal Y_{p,q,r}}
&\asymp\|M_{q,\delta}f\|_{\mathcal Y_{p,q,r}}.
\end{align*}
\end{proof}

The following example displays the dependence of the local exponent on the
summing index $r$.

\begin{example}\label{ex:concrete-indices}
Fix $\delta>0$ and let $f\in\mathcal D_\varphi\cap L^4_{\loc}$.  For
$p=q=4$,
\eqref{eq:kappa-diagonal-intro} gives
\[
\kappa(4,4,1)=\frac43,
\qquad
\kappa(4,4,2)=2,
\qquad
\kappa(4,4,6)=4.
\]
Hence
\begin{align*}
G_{4,\delta}f\in L^{4/3}
&\Longrightarrow
\left[
T_f\in\Piabs_1(F_\varphi^4,F_\varphi^4)
\iff M_{4,\delta}f\in L^{4/3}
\right],\\
G_{4,\delta}f\in L^2
&\Longrightarrow
\left[
T_f\in\Piabs_2(F_\varphi^4,F_\varphi^4)
\iff M_{4,\delta}f\in L^2
\right],\\
G_{4,\delta}f\in L^4
&\Longrightarrow
\left[
T_f\in\Piabs_6(F_\varphi^4,F_\varphi^4)
\iff M_{4,\delta}f\in L^4
\right].
\end{align*}
The target exponent changes from $4/3$ to $2$ and then to $4$.

Now let $f\in\mathcal D_\varphi\cap L^4_{\loc}$ and
$G_{4,\delta}f\in L^\infty$.  Since
\[
\kappa(1,4,r)=\infty,
\qquad 1\le r<\infty,
\]
\cref{cor:complete-parameter} gives
\[
T_f\in\Piabs_r(F_\varphi^1,F_\varphi^4)
\iff M_{4,\delta}f\in L^\infty,
\qquad 1\le r<\infty.
\]
\end{example}

The last line is important.  It converts the former exceptional diagonal
range into an explicit Toeplitz criterion.  If $(p,q,r)\in\mathcal E$ and
\[
\kappa=\frac{p'q(r-2)}{(p'-2)(q-2)+2(r-2)},
\]
then, under the matching condition $G_{q,\delta}f\in L^\kappa$,
\[
T_f\in\Piabs_r(F_\varphi^p,F_\varphi^q)
\quad\Longleftrightarrow\quad
M_{q,\delta}f\in L^\kappa.
\]
The same statement holds with $M_{q,\delta}f$ replaced by
$a_\delta f$ or $\widetilde f$.  This is a consequence of the completed
diagonal index in \cite{FanHeWangZeng2026} and of the Fock-space graph and
scalar-recovery results proved here.

The same-exponent case reduces to the explicit index in
\eqref{eq:kappa-diagonal-intro}.

\begin{corollary}\label{cor:equal-exponents}
Let $1\le p<\infty$, $1\le r<\infty$, fix $\delta>0$, and put
$\kappa(p,r)=\kappa(p,p,r)$ as in \eqref{eq:kappa-diagonal-intro}.  If
\[
f\in\mathcal D_\varphi\cap L^p_{\loc},
\qquad
G_{p,\delta}f\in L^{\kappa(p,r)},
\]
then
\[
T_f\in\Piabs_r(F_\varphi^p,F_\varphi^p)
\quad\Longleftrightarrow\quad
M_{p,\delta}f\in L^{\kappa(p,r)}.
\]
Equivalently, $\widetilde f\in L^{\kappa(p,r)}$, and
\[
\pi_r(T_f)+\|G_{p,\delta}f\|_{L^{\kappa(p,r)}}
\asymp
\|M_{p,\delta}f\|_{L^{\kappa(p,r)}}.
\]
\end{corollary}

\begin{proof}
For $p=q$, the triple $(p,p,r)$ does not belong to $\mathcal E$.
Equations \eqref{eq:kappa-diagonal-intro} and
\eqref{eq:garling-identification} give
\[
\mathfrak d_{p,p}^{\,r}=\ell^{\kappa(p,r)}.
\]
Proposition~\ref{prop:explicit-realization} yields
\begin{align*}
\|M_{p,\delta}f\|_{\mathfrak D_{p,p}^{\,r}}
&\asymp\|M_{p,\delta}f\|_{L^{\kappa(p,r)}},\\
\|G_{p,\delta}f\|_{\mathfrak D_{p,p}^{\,r}}
&\asymp\|G_{p,\delta}f\|_{L^{\kappa(p,r)}}.
\end{align*}
The operator equivalence follows from \cref{cor:intro-ida}.  The equivalent Berezin-transform formulation follows from
\cref{cor:explicit-scalar-recovery}.
\end{proof}

At $p=q=2$, the result agrees with the Hilbert--Schmidt criterion.

\begin{corollary}
Let $S_2$ be the Hilbert--Schmidt ideal.  Thus, for an orthonormal basis
$(e_j)$, an operator $A$ belongs to $S_2$ when
$\sum_j\|Ae_j\|^2<\infty$; its norm is the square root of this sum.
Fix $\delta>0$, let
$f\in\mathcal D_\varphi\cap L^2_{\loc}$, and assume that
$G_{2,\delta}f\in L^2$.
For every $1\le r<\infty$, the following are equivalent:
\[
T_f\in\Piabs_r(F_\varphi^2,F_\varphi^2),
\qquad
T_f\in S_2(F_\varphi^2),
\qquad
M_{2,\delta}f\in L^2.
\]
Moreover,
\[
\pi_r(T_f)+\|G_{2,\delta}f\|_2
\asymp
\|T_f\|_{S_2}+\|G_{2,\delta}f\|_2
\asymp
\|M_{2,\delta}f\|_2.
\]
\end{corollary}

\begin{proof}
For $p=q=2$, \eqref{eq:kappa-diagonal-intro} gives
$\kappa(2,2,r)=2$.  Moreover, for Hilbert spaces $H_1,H_2$,
\begin{equation}\label{eq:hilbert-pi-S2}
\Piabs_r(H_1,H_2)=S_2(H_1,H_2),
\qquad
\pi_r(A)\asymp\|A\|_{S_2}.
\end{equation}
The constants in this standard Hilbert-space identity may depend on $r$;
see \cite{DiestelJarchowTonge1995}.
Apply \cref{cor:equal-exponents} and then
\eqref{eq:hilbert-pi-S2}.
\end{proof}

Order continuity also gives compactness for the Toeplitz--Hankel graph.

\begin{corollary}
Let $1\le p,q,r<\infty$ and fix $\delta>0$.
Assume that $c_{00}$ is dense in $\xk$ and that
\[
f\in\mathcal D_\varphi\cap L^q_{\loc},
\qquad
G_{q,\delta}f\in\Xk.
\]
If $T_f\in\Piabs_r(F_\varphi^p,F_\varphi^q)$, then
\begin{equation*}
M_f:F_\varphi^p\to L_\varphi^q,
\qquad
T_f:F_\varphi^p\to F_\varphi^q,
\qquad
H_f:F_\varphi^p\to L_\varphi^q
\end{equation*}
are compact.  In particular, this holds for every parameter triple for
which $\kappa(p,q,r)\ne\infty$.
\end{corollary}

\begin{proof}
By \cref{thm:hankel},
\[
H_f\in\Piabs_r,
\qquad
\pi_r(H_f)\asymp\|G_{q,\delta}f\|_{\Xk}.
\]
Hence \eqref{eq:multiplication-on-core} and the ideal triangle inequality give
\[
M_f\in\Piabs_r,
\qquad
\pi_r(M_f)\le\pi_r(T_f)+\pi_r(H_f).
\]
Corollary~\ref{cor:multiplication-compact} shows that $M_f$ is compact.  Since
\[
T_f=PM_f,
\qquad
H_f=(I-P)M_f,
\]
the other two operators are compact.
\end{proof}

For analytic symbols, the IDA term vanishes.  The main theorem then forces
rigidity.

\begin{corollary}
Let $1\le p,q,r<\infty$, put $\kappa=\kappa(p,q,r)$, and let
$f\in\mathcal D_\varphi$ be entire.  Then
\begin{equation*}
T_f\in\Piabs_r(F_\varphi^p,F_\varphi^q)
\quad\Longleftrightarrow\quad
\begin{cases}
f=0,&\kappa\ne\infty,\\
f\text{ is constant},&\kappa=\infty.
\end{cases}
\end{equation*}
\end{corollary}

\begin{proof}
Since $f$ is entire,
\[
G_{q,\delta}f=0,
\qquad
H_f=0,
\qquad
M_f=\iota T_f.
\]
Hence \cref{thm:intro-main,prop:explicit-realization} give
\begin{equation}\label{eq:analytic-local-criterion}
T_f\in\Piabs_r
\iff M_{q,\delta}f\in\mathfrak E_\kappa.
\end{equation}
If $\kappa=s<\infty$, then $M_{q,\delta}f\in L^s$.  The holomorphic
submean inequality gives
\[
|f(z)|\le M_{q,\delta}f(z),
\]
and hence $f\in L^s(\C^n)$.  For every $R>0$,
\[
|f(z)|^s
\le \frac{C}{R^{2n}}
\int_{B(z,R)}|f(w)|^s\dd v(w)
\le \frac{C}{R^{2n}}\|f\|_{L^s}^s.
\]
Letting $R\to\infty$ gives $f(z)=0$.

If $\kappa=q^-$, the chosen Young function satisfies
$\Psi_q(t)\ge c t^q$ for $t\ge0$.  Thus
\[
M_{q,\delta}f\in L^{q^-}
\quad\Longrightarrow\quad
f\in L^q,
\]
and the same argument gives $f=0$.  If $\kappa=\infty$, then
$M_{q,\delta}f\in L^\infty$, so $f$ is bounded and therefore constant.
The converse implications follow from \eqref{eq:analytic-local-criterion}.
\end{proof}

At the opposite extreme from analytic symbols, compact support gives an
automatic summability criterion.

\begin{corollary}
Let $1\le p,q,r<\infty$, fix $\delta>0$, and let
$f\in L^q(\C^n)$ have compact support.
Then
\[
M_f,\quad T_f,\quad H_f
\]
are $r$-summing and compact.  Let $S=\supp f$, where $\supp f$ is the
support of $f$, and define
\[
\dist(z,S)=\inf_{w\in S}|z-w|,
\qquad
S_\delta=\{z:\dist(z,S)<\delta\}.
\]
Then
\begin{equation}\label{eq:compact-support-bound}
\pi_r(M_f)+\pi_r(T_f)+\pi_r(H_f)
\lesssim
\|f\|_{L^q}\,
\|\mathbf1_{J(S_\delta)}\|_{\mathfrak d_{p,q}^{\,r}},
\end{equation}
where $J(S_\delta)$ is the finite set of reference cells meeting
$S_\delta$.
\end{corollary}

\begin{proof}
For every $z\in\C^n$, H\"older's inequality on $S$ and local boundedness
of $k_z e^{-\varphi}$ give
\[
\int_S|f(w)k_z(w)|e^{-\varphi(w)}\dd v(w)<\infty.
\]
Thus $f\in\mathcal D_\varphi$.
The local size satisfies
\[
M_{q,\delta}f(z)=0\quad(z\notin S_\delta),
\qquad
M_{q,\delta}f(z)\le |B(0,\delta)|^{-1/q}\|f\|_{L^q}.
\]
Therefore
\[
\|M_{q,\delta}f\|_{\Xk}
\lesssim
\|f\|_{L^q}\,
\|\mathbf1_{J(S_\delta)}\|_{\mathfrak d_{p,q}^{\,r}}<\infty.
\]
Theorem~\ref{thm:intro-main} gives the summing assertions and
\eqref{eq:compact-support-bound}.  To prove compactness, let
$\{g_m\}$ be bounded in $F_\varphi^p$.  By Montel's theorem, a subsequence
converges locally uniformly.  Since $f$ is supported on $S$,
\begin{align*}
\|M_fg_m-M_fg\|_{q,\varphi}^q
&=\int_S|f|^q|g_m-g|^qe^{-q\varphi}\dd v\\
&\le
\|f\|_{L^q(S)}^q
\sup_S\bigl(|g_m-g|^qe^{-q\varphi}\bigr)\longrightarrow0.
\end{align*}
Thus $M_f$ is compact.  The identities
$T_f=PM_f$ and $H_f=(I-P)M_f$ complete the proof.
\end{proof}

We next remove the matching IDA assumption for symbols whose values remain
in a fixed open sector.

\begin{corollary}[Sectorial symbols]\label{cor:sectorial}
Let $1\le p,q,r<\infty$.  Suppose that
$f\in\mathcal D_\varphi\cap L^q_{\loc}$.  Assume that there exist
$\eta\in\C$ and $0\le\theta<\pi/2$ such that $|\eta|=1$ and
\begin{equation}\label{eq:sectorial-range}
\operatorname{Re}(\eta f(z))\ge(\cos\theta)|f(z)|
\quad\text{for a.e. }z\in\C^n.
\end{equation}
For a fixed sufficiently small $\delta>0$, assume also that
there is a constant $C_f>0$ such that
\begin{equation}\label{eq:local-reverse-holder}
M_{q,\delta}f(z)\le C_f\,a_\delta(|f|)(z),
\qquad z\in\C^n,
\end{equation}
with $C_f$ independent of $z$.  Then, without an IDA hypothesis,
\[
\begin{aligned}
T_f\in\Piabs_r(F_\varphi^p,F_\varphi^q)
&\quad\Longleftrightarrow\quad M_{q,\delta}f\in\Xk\\
&\quad\Longleftrightarrow\quad a_\delta f\in\Xk
\quad\Longleftrightarrow\quad\widetilde f\in\Xk.
\end{aligned}
\]
Moreover,
\[
\pi_r(T_f)\asymp\|M_{q,\delta}f\|_{\Xk}
\asymp\|a_\delta f\|_{\Xk}
\asymp\|\widetilde f\|_{\Xk},
\]
where the constants may depend on $\theta$ and $C_f$.

If $q=1$, \eqref{eq:local-reverse-holder} holds with $C_f=1$.  Hence every
sectorial $f\in\mathcal D_\varphi\cap L^1_{\loc}$ satisfying
\eqref{eq:sectorial-range} obeys
\[
T_f\in\Piabs_r(F_\varphi^p,F_\varphi^1)
\iff M_{1,\delta}f\in L^{s_p}
\iff a_\delta f\in L^{s_p}
\iff \widetilde f\in L^{s_p},
\]
where $s_p$ is defined by \eqref{eq:q-one-exponent-intro}.  Moreover,
\begin{equation*}
\pi_r(T_f)
\asymp\|M_{1,\delta}f\|_{L^{s_p}}
\asymp\|a_\delta f\|_{L^{s_p}}
\asymp\|\widetilde f\|_{L^{s_p}}.
\end{equation*}
\end{corollary}

\begin{proof}
\emph{Necessity.}
Suppose first that $T_f$ is $r$-summing.  By
\cref{cor:scalar-necessity},
\begin{equation}\label{eq:sectorial-necessity}
\|\widetilde f\|_{\Xk}\lesssim\pi_r(T_f).
\end{equation}
The Berezin kernel is positive.  Hence
\begin{align*}
|\widetilde f(z)|
&\ge \operatorname{Re}\bigl(\eta\widetilde f(z)\bigr)\\
&=\int_{\C^n}\operatorname{Re}\bigl(\eta f(w)\bigr)
|k_z(w)|^2e^{-2\varphi(w)}\dd v(w)\\
&\ge(\cos\theta)
\int_{\C^n}|f(w)||k_z(w)|^2e^{-2\varphi(w)}\dd v(w).
\end{align*}
By the choice of $\delta$, \eqref{eq:kernel-lower} holds on
$B(z,\delta)$.  Hence
\begin{equation}\label{eq:sectorial-average-lower}
|\widetilde f(z)|
\gtrsim(\cos\theta)a_\delta(|f|)(z).
\end{equation}
Equations \eqref{eq:local-reverse-holder},
\eqref{eq:sectorial-average-lower}, and
\eqref{eq:sectorial-necessity} give
\begin{equation*}
\|M_{q,\delta}f\|_{\Xk}
\lesssim_{C_f,\theta}\|\widetilde f\|_{\Xk}
\lesssim\pi_r(T_f).
\end{equation*}

\emph{Sufficiency.}
Conversely, \cref{cor:one-sided} gives
\begin{equation*}
M_{q,\delta}f\in\Xk
\Longrightarrow T_f\in\Piabs_r,
\qquad
\pi_r(T_f)\lesssim\|M_{q,\delta}f\|_{\Xk}.
\end{equation*}
\emph{Comparison of the scalar quantities.}
It remains to compare the scalar functions.  First,
\begin{equation*}
|a_\delta f|\le a_\delta(|f|)\le M_{q,\delta}f.
\end{equation*}
On the other hand, \eqref{eq:sectorial-range} gives
\begin{equation*}
|a_\delta f|
\ge \operatorname{Re}(\eta a_\delta f)
\ge(\cos\theta)a_\delta(|f|)
\ge\frac{\cos\theta}{C_f}M_{q,\delta}f.
\end{equation*}
Thus $|a_\delta f|\asymp M_{q,\delta}f$ pointwise.  Finally, if
$z\in E_j$, \eqref{eq:kernel-upper} and H\"older's inequality give
\[
|\widetilde f(z)|
\lesssim
\sum_m e^{-c|j-m|}M_{q,R}f(m).
\]
Equation~\eqref{eq:convolution-Xk} and radius independence yield
\begin{equation*}
\|\widetilde f\|_{\Xk}
\lesssim\|M_{q,\delta}f\|_{\Xk}.
\end{equation*}
The reverse estimate follows from
\eqref{eq:sectorial-average-lower} and
\eqref{eq:local-reverse-holder}.  Since
$M_{1,\delta}f=a_\delta(|f|)$, the final assertion follows with $C_f=1$.
\end{proof}

The sector condition permits genuinely variable complex phases.

\begin{example}
Let $g\ge0$ be measurable.  Let $\psi:\C^n\to\mathbb R$ satisfy
\[
|\psi(z)-\psi_0|\le\theta<\frac\pi2
\quad\text{a.e.},
\qquad
f(z)=g(z)e^{i\psi(z)}.
\]
Then \eqref{eq:sectorial-range} holds with $\eta=e^{-i\psi_0}$.  If
$f\in\mathcal D_\varphi\cap L^1_{\loc}$, \cref{cor:sectorial} gives
\[
T_f\in\Piabs_r(F_\varphi^p,F_\varphi^1)
\iff M_{1,\delta}f\in L^{s_p}
\iff\widetilde f\in L^{s_p}.
\]
The phase $\psi$ need not be constant or continuous.
\end{example}

\begin{remark}
The positive case is $\theta=0$ and $\eta=1$.  For $q>1$, positivity alone
does not compare local $L^q$ and $L^1$ size.  Thus it does not remove
matching IDA.  Condition \eqref{eq:local-reverse-holder} is precisely the
extra comparison used above.  For $q=1$, no extra local regularity is
needed.
\end{remark}

\begin{remark}
Theorem~\ref{thm:intro-main} characterizes $(T_f,H_f)$ without IDA.
For $T_f$ alone, \cref{cor:intro-ida} uses
\[
H_f\in\Piabs_r
\iff G_{q,\delta}f\in\Xk.
\]
Without this information, $M_f=\iota T_f+H_f$ gives no estimate for
$M_f$.  The sectorial class in \cref{cor:sectorial} is an exception.
Removing matching IDA for unrestricted complex symbols remains open.
\end{remark}

\subsection{Sharpness and model symbols}\label{sec:sharpness}

The diagonal ideal is not merely an artifact of the proof.  It is already forced by elementary positive symbols supported on separated balls.

\begin{theorem}[Sharp separated models]\label{thm:sharpness}
Fix $1\le p,q,r<\infty$.  Let $\Lambda=\{a_j\}$ be sufficiently
separated.  Choose $\rho>0$ so that the balls $B(a_j,4\rho)$ are pairwise
disjoint and \eqref{eq:kernel-lower} holds on $B(a_j,\rho)$.  For
$c=(c_j)\in\ell^\infty$, $c_j\ge0$, define
\begin{equation*}
f_c(z)=\sum_jc_j\mathbf1_{B(a_j,\rho)}(z).
\end{equation*}
Then $f_c\in\mathcal D_\varphi\cap L^q_{\loc}$ and
\begin{equation}\label{eq:sharp-equivalence}
T_{f_c}\in\Piabs_r(F_\varphi^p,F_\varphi^q)
\quad\Longleftrightarrow\quad
c\in\xk.
\end{equation}
In that case,
\begin{equation}\label{eq:sharp-norm}
\pi_r(T_{f_c})\asymp\|c\|_{\xk}.
\end{equation}
Thus the diagonal ideal $\mathfrak d_{p,q}^{\,r}$ cannot be replaced by a smaller sequence lattice, even for nonnegative symbols.
\end{theorem}

\begin{proof}
\emph{The local norm.}
The symbol is bounded.  Hence $f_c\in\mathcal D_\varphi$.  Since the
enlarged balls are disjoint, $B(z,2\rho)$ meets at most one support ball.
If $z\in B(a_j,\rho/2)$, then
$B(a_j,\rho)\subset B(z,2\rho)$.  Therefore
\[
M_{q,2\rho}f_c(z)\lesssim
\sum_jc_j\mathbf1_{B(a_j,3\rho)}(z),
\qquad
M_{q,2\rho}f_c(z)\gtrsim
\sum_jc_j\mathbf1_{B(a_j,\rho/2)}(z).
\]
Let $u=M_{q,2\rho}f_c$ and let $u_k^\#$ be its local suprema on the reference cells.  The first pointwise estimate gives
\[
u_k^\#
\lesssim
\sum_{a_j\in\mathcal N(k)}c_j,
\]
where the neighbor relation has uniformly bounded degree in both variables.  Conversely, if $E_{k(j)}$ contains $a_j$, then the second estimate gives
\[
c_j\lesssim u_{k(j)}^\#.
\]
The map $j\mapsto k(j)$ has uniformly bounded fibers.  Decompose it into finitely many injections.  Symmetry, solidity, and
\cref{lem:sequence-tail} now give
\begin{equation}\label{eq:bump-local-norm}
\|M_{q,2\rho}f_c\|_{\Xk}\asymp\|c\|_{\xk}.
\end{equation}

\emph{Sufficiency and necessity.}
If $c\in\xk$, \cref{cor:one-sided} and \eqref{eq:bump-local-norm} give
$\pi_r(T_{f_c})\lesssim\|c\|_{\xk}$.  Conversely, positivity and the lower kernel estimate give
\[
\widetilde {f_c}(a_j)
=\int_{\C^n}f_c(w)|k_{a_j}(w)|^2e^{-2\varphi(w)}\dd v(w)
\gtrsim c_j.
\]
If $T_{f_c}$ is $r$-summing, \eqref{eq:operator-diagonal-discrete} yields
\[
\|c\|_{\xk}
\lesssim\|\{\widetilde {f_c}(a_j)\}\|_{\xk}
\lesssim\pi_r(T_{f_c}).
\]
This proves \eqref{eq:sharp-equivalence} and \eqref{eq:sharp-norm}.
\end{proof}

The same construction has a sharp complex sectorial form at $q=1$.

\begin{example}
Let $\Lambda=\{a_j\}$ and $\rho$ satisfy the assumptions of
\cref{thm:sharpness}.  Let $c=(c_j)\in\ell^\infty$, $c_j\ge0$, and choose
real numbers $\vartheta_j$ such that
\[
|\vartheta_j-\vartheta_0|\le\theta<\frac\pi2.
\]
Define
\begin{equation*}
f_{c,\vartheta}(z)
=\sum_jc_je^{i\vartheta_j}\mathbf1_{B(a_j,\rho)}(z).
\end{equation*}
Then $f_{c,\vartheta}\in\mathcal D_\varphi\cap L^1_{\loc}$ and is
sectorial with $\eta=e^{-i\vartheta_0}$.  Furthermore,
\begin{equation*}
T_{f_{c,\vartheta}}\in
\Piabs_r(F_\varphi^p,F_\varphi^1)
\quad\Longleftrightarrow\quad
c\in\ell^{s_p},
\qquad
\pi_r(T_{f_{c,\vartheta}})\asymp\|c\|_{\ell^{s_p}}.
\end{equation*}
The constants are uniform in $(\vartheta_j)$ subject to the displayed
sector condition.  The model therefore permits nonconstant,
discontinuous complex phases.
\end{example}

\begin{proof}
Boundedness gives $f_{c,\vartheta}\in\mathcal D_\varphi$.  Since the
support balls are disjoint,
\[
|f_{c,\vartheta}|=
\sum_jc_j\mathbf1_{B(a_j,\rho)}.
\]
The proof of \eqref{eq:bump-local-norm}, with $q=1$, gives
\begin{equation*}
\|M_{1,2\rho}f_{c,\vartheta}\|_{L^{s_p}}
\asymp\|c\|_{\ell^{s_p}}.
\end{equation*}
Now apply \cref{cor:sectorial} and radius independence.  The constants
depend on the phases only through $\cos\theta$.
\end{proof}

The separated model also detects the logarithmic boundary case.

\begin{remark}
Take $p=3$, $q=3/2$, and $r=1$.  Then $\kappa=q^-$.  Theorem~\ref{thm:sharpness} gives
\[
T_{f_c}\in\Piabs_1(F_\varphi^3,F_\varphi^{3/2})
\quad\Longleftrightarrow\quad
c\in\ell^{(3/2)^-}.
\]
Thus the logarithmic Orlicz condition cannot be replaced by
$c\in\ell^{3/2}$.
\end{remark}

\begingroup
\renewcommand{\addcontentsline}[3]{}
\section*{Funding}
The first author was supported by the National Natural Science Foundation of
China (Grant No.~12401154).  The second author was supported by the National
Natural Science Foundation of China (Grant No.~12601234).

\section*{Data availability}
No data were used for the research described in this article.

\section*{Declaration of competing interest}
The authors declare that they have no known competing financial interests or
personal relationships that could have appeared to influence the work
reported in this paper.
\endgroup

\end{document}